\documentclass[a4paper, 11pt, reqno]{amsart} 
\usepackage[normalem]{ulem}
\usepackage{amsmath, amsthm, amssymb, tikz} 
\usepackage{bbm}
\usepackage[utf8]{inputenc} 
\usepackage{graphicx} 
\usepackage[small]{caption}
\usetikzlibrary{arrows,calc}
\numberwithin{equation}{section}
\usepackage{stmaryrd}
\usepackage{mathtools}

\usepackage{lmodern} 
\usepackage[T1]{fontenc} 

\DeclareFontShape{OMX}{cmex}{m}{n}{
  <-7.5> cmex7
  <7.5-8.5> cmex8
  <8.5-9.5> cmex9
  <9.5-> cmex10
}{}
\SetSymbolFont{largesymbols}{normal}{OMX}{cmex}{m}{n}
\SetSymbolFont{largesymbols}{bold}  {OMX}{cmex}{m}{n}

\usepackage{enumerate} 
\usepackage{verbatim} 
\usepackage{color}
\usepackage{tabularx}

\usepackage[colorlinks, linkcolor=blue]{hyperref} 

\usepackage{orcidlink}

\usepackage[letterpaper,nomarginpar]{geometry}

\usepackage[capitalise]{cleveref}

\newcommand{\R}{\mathbb{R}}
\newcommand{\E}{\mathbb{E}}

\newcommand{\N}{\mathbb{N}}

\newcommand{\cB}{\mathcal B}

\newcommand{\cF}{\mathcal F}
\newcommand{\cN}{\mathcal{N}}
\newcommand{\cT}{\mathcal{T}}

\newcommand{\cS}{\mathcal{S}}

\renewcommand{\d}{\mathrm{d}}

\renewcommand{\P}{\mathbb{P}}

\newcommand{\ind}[1]{\mathbbm{1}_{\{ #1\}}}
\newcommand{\indset}[1]{\mathbbm{1}_{#1}}

\newcommand{\ep}{\varepsilon}

\newtheorem{maintheorem}{Theorem}[section]

\newtheorem{theorem}{Theorem}[section]

\newtheorem*{theorem*}{Theorem}
\newtheorem{lemma}[theorem]{Lemma}

\newtheorem{proposition}[theorem]{Proposition}

\theoremstyle{definition}{

\newtheorem*{definition*}{Definition}

\newtheorem*{question*}{Question}
\newtheorem*{example*}{Example}
\newtheorem*{examples*}{Examples}
\newtheorem{remark}[theorem]{Remark}
\newtheorem*{remark*}{Remark}

\newtheorem*{claim*}{Claim}
}

\renewcommand{\bar}[1]{\overline{#1}}

\newcommand{\Cov}{\operatorname{Cov}}
\newcommand{\Var}{\operatorname{Var}}

\newcommand{\given}{\;\big|\;}

\usepackage[backend = biber, sorting = nyt, defernumbers=true, style = numeric, eprint = true, doi = false, isbn = false, date = year, texencoding = ascii]{biblatex}
\bibliography{LitGMC}

\newcommand{\ao}{a^{(1)}}
\newcommand{\at}{a^{(2)}}
\newcommand{\ai}{a^{(i)}}

\newcommand{\gi}{g^{(i)}}

\newcommand{\Co}[3]{C_{#1,#2}^{(1)}(#3)}
\newcommand{\Ct}[3]{C_{#1,#2}^{(2)}(#3)}
\newcommand{\Ci}[3]{C_{#1,#2}^{(i)}(#3)}

\newcommand{\ta}{\widetilde{a}}
\newcommand{\tg}{\widetilde{g}}
\newcommand{\tS}{\widetilde{S}}
\newcommand{\tZ}{\widetilde{Z}}

\title{Absolute continuity of non-Gaussian and Gaussian multiplicative chaos measures}
\author[Y. H. Kim]{Yujin H. Kim \orcidlink{0000-0001-8457-2548}}
\address{Y.\ H.\ Kim\hfill\break
The Division of Physics, Mathematics and Astronomy \\ 
California Institute of Technology\\
1201 E California Blvd \\ Pasadena, CA 91125, USA.}
\email{yujin@caltech.edu}
 \author[X. Kriechbaum]{Xaver Kriechbaum \orcidlink{0009-0005-4470-299X}}
\address{X.\ Kriechbaum\hfill\break
Department of Mathematics\\
Weizmann Institute of Science\\
234 Herzl Street\\
Rehovot 76100, Israel\break
\hphantom{ad}\emph{Current address:}\hfill\break
Institut de Math\'{e}matiques de Toulouse\\
Universit\'{e} de Toulouse\\
118 route de Narbonne\\
31062 Toulouse Cedex 9, France}
\email{xaver.kriechbaum@math.univ-toulouse.fr}

\allowdisplaybreaks
\begin{document}

\begin{abstract}
In this article, we consider the multiplicative chaos measure associated to the log-correlated random Fourier series, or random wave model, with i.i.d.\ coefficients taken from a general class of distributions. This measure was shown to be non-degenerate when the inverse temperature is subcritical by Junnila (Int.\ Math.\ Res.\ Not.\ \textbf{2020} (2020), no.\ 19, 6169-6196). When the coefficients are Gaussian, this measure is an example of a Gaussian multiplicative chaos (GMC), a well-studied universal object in the study of log-correlated fields.
In the case of non-Gaussian coefficients, the resulting chaos is \emph{not} a GMC in general. However, we construct a coupling between the non-Gaussian multiplicative chaos measure and a GMC such that the two are almost surely mutually absolutely continuous.  
\end{abstract}

{\mbox{}
\maketitle
}
\vspace{-.5cm}

\section{Introduction}
This article considers the multiplicative chaos measures associated to non-Gaussian log-correlated fields and their shared properties with respect to \emph{Gaussian multiplicative chaos (GMC)}, formalized through the relation of absolute continuity. 
The model we analyze is the log-correlated random Fourier series, or random wave model. We focus on the space $[0,1]$ for ease of exposition and address general bounded domains in $\R^d$, $d\geq 1$, in \cref{subsec:general-bounded-domains}.
We begin by defining the random Fourier series and stating the main theorem. After this, we review GMC, discuss related literature in the non-Gaussian setting, and provide an outline of the proof of the main theorem.

Consider the random Fourier series 
\begin{align}\label{def:trig-field}
    S_{n,a}(t) := \sum_{k=1}^n \frac{1}{\sqrt{k}} \Big( a_k^{(1)} \cos(2\pi k t) + a_k^{(2)} \sin(2\pi kt)\Big)\,, \quad t\in[0,1]\,,
\end{align}
where $a:= (a_k^{(i)})_{k\in \N, i\in\{1,2\}}$ denotes a sequence of i.i.d.\ random variables such that
\begin{equation}\label{eq:CondsOnak}
\mathbb{E}[\ao_1] = 0,\quad \Var[\ao_1] = 1,\quad\text{and}\quad \mathbb{E}[e^{\lambda \ao_1}]<\infty\ \text{for all}\ \lambda\in\R \,.
\end{equation} 
Now, consider the measure on $[0,1]$, with inverse temperature $\gamma>0$, defined by
\begin{equation} \label{eq:PreLimGMC}
    \mu_{n,\gamma,a}(\d t) := \frac{e^{\gamma S_{n,a}(t)}}{Z_{n,\gamma,a}(t)} \, \d t\,,
\end{equation}
where  $Z_{n,\gamma,a}(t) := \mathbb{E}\left[e^{\gamma S_{n,a}(t)}\right]$  so that $\mu_{n,\gamma,a}[A]$ is a martingale with respect to the natural filtration of $a$ for all Borel sets $A$.

In \cite{Junnila}, Junnila showed that for $\gamma \in (0,\sqrt2)$  there exists a non-degenerate measure $\mu_{\gamma,a}$ such that  $\mu_{n,\gamma,a}$ converges 
almost surely
to  $\mu_{\gamma,a}$ as $n\to\infty$ in the space of Radon measures on $[0,1]$ equipped with the topology of weak convergence (for $\gamma \geq \sqrt2$, the limiting measure is degenerate).
We call $\mu_{\gamma,a}$ the \emph{multiplicative chaos}  measure associated to the fields $(S_{n,a})_{n\in\N}$.

When $a= g$ is a sequence of i.i.d.\ standard Gaussians, the non-degeneracy of $\mu_{\gamma,g}$ has been known since Kahane \cite{kahanegmc}, and $\mu_{\gamma,g}$ is called Gaussian multiplicative chaos (GMC). 
GMC is a fractal random measure with fascinating properties and many important applications in modern probability; we review this object in Section~\ref{Sec:GMCBackGround} below.
Here, we mention that GMC is a universal object expected to appear when exponentiating a field that exhibits asymptotically Gaussian behavior and covariance diverging logarithmically on short scales.

While the field $(S_{n,a}(t))_{t\in[0,1]}$ is log-correlated (shown in \cite[Section~3.1]{Junnila}),
the chaos measure $\mu_{\gamma,a}$ is \emph{not necessarily} a GMC in general
due to the contribution of, for instance, the first (non-Gaussian) term in the sum defining $S_{n,a}(t)$. This begs the question: \textit{for general sequences $a$, what properties does $\mu_{\gamma,a}$  share with GMC?}

Our main result addresses this question through the lens of absolute continuity.

\begin{maintheorem}\label{Theo:Main}
Let $\gamma\in (1,\sqrt{2})$ and take $(\ai_k)_{i\in\{1,2\},k\in\N}$ i.i.d.\@ satisfying \cref{eq:CondsOnak}. There exists a sequence of i.i.d.\@  standard  Gaussian random variables $g:= (g_k^{(i)})_{i\in\{1,2\},k\in\N}$ coupled to $a:= (\ai_k)_{i\in\{1,2\},k\in\N}$ such that 
\[
\mu_{\gamma,g} \ll \mu_{\gamma,a} \ll\mu_{\gamma,g} \qquad\text{almost surely.}
\]
\end{maintheorem}
The $(\gi_k)_{k\in\N, i\in\{1,2\}}$ are defined in \eqref{def:iid-gaussians}, which makes use of \eqref{Def:Coeffs}, \eqref{def:time-n-thick} and Lemma~\ref{Lem:CouplThick}. We expect \cref{Theo:Main} can assist in proving a variety of almost-sure properties of both $\mu_{\gamma,a}$ and the associated approximating sequence of non-Gaussian fields $(S_{n,a})_{n\in\N}$ shared by the GMC $\mu_{\gamma,g}$ and the Gaussian log-correlated fields $(S_{n,g})_{n\in\N}$, for $\gamma \in (1,\sqrt2)$.

\begin{remark}[Higher dimensions]\label{rk:higher-dim}
In \cref{thm:general}, we extend our result to random Fourier series in bounded domains in $\R^d$. 
\end{remark}

\begin{remark}[The result holds in the full subcritical regime on unit cubes]
Our proof uses as a blackbox a high-dimensional coupling result dating back to Yurinskii \cite{yurinskii1978error}. Sometime after the posting of this paper, a new high-dimensional coupling was developed in \cite{ChowGang}. They verify that replacing the Yurinskii coupling with the coupling of \cite{ChowGang} in our proof and then following the steps of our proof exactly  yields \cref{Theo:Main} in the full subcritical regime $(0,\sqrt{2})$. In regards to higher dimensions (see \cref{rk:higher-dim}), they show that for unit cube domains $D = (0,1]^d$, their coupling tool applies in the same way to yield \cref{thm:general} in the full subcritical regime $(0,\sqrt{2d})$. In \cref{subsec:proof-outline}, we discuss in more detail why our application of the Yurinskii coupling requires $\gamma >\sqrt{d}$. See Remark~\ref{rk:full-subcritical-higher-dim} for more on the restriction of \cite{ChowGang} to unit cubes.
\end{remark}

\subsection{Background on GMC}\label{Sec:GMCBackGround}
Gaussian multiplicative chaos was introduced by Kahane in \cite{kahanegmc} as a mathematical model for energy dissipation in turbulence, making rigorous a program initiated by Mandelbrot in \cite{mandelbrot-turbulence}.
In \cite{kahanegmc}, considering a locally compact metric space $(T,\rho)$ with reference (Radon) measure $\sigma$, Kahane gave a rigorous interpretation of the measure
\begin{equation} \label{eq:ExpMeasure}
  M_{\gamma}(\d x) := e^{\gamma X(x)-\frac{\gamma^2}{2}\Var[X(x)]} \ \sigma(\mathrm{d}x) \,,
\end{equation}
where $(X(x))_{x\in T}$ formally denotes a Gaussian field with $\sigma$-positive covariance kernel
\begin{equation} \label{eq:LogKernel}
  R(x,y) = \max\{0, \log(\rho(x,y)^{-1})\}+g(x,y)\,,
\end{equation}
for $g$ bounded and continuous.
The kernel $R$ is of $\sigma$-positive type if there exists a sequence $(R_i)_i$ of continuous, pointwise nonnegative, and positive definite kernels $R_i : T^2 \to \R_+$ with \begin{equation}\label{eq:sigmaType}
R = \sum_i R_i \,. 
\end{equation} 
Formally, \cref{eq:sigmaType} means that the Gaussian field $S$ can be approximated by the (pointwise-defined) Gaussian fields $S_n = \sum_{i=1}^n X_i$, where $X_i$ is an independent centered Gaussian with covariance $R_i$. 
The measure $M_{\gamma}$ in~\cref{eq:ExpMeasure} is then defined as the limit of the martingale sequence
\[
M_{n,\gamma}(\d x) := e^{\gamma S_n(x)-\frac{\gamma^2}{2}\Var[S_n(x)]} \ \sigma(\d x)\,.
\]
 An important property proven in \cite{kahanegmc} is that the law of $M_{\gamma}$
is independent of the sequence $(S_n)_{n\in\N}$ approximating $S$, so that we can associate a \emph{unique} measure $M_{\gamma}$, called GMC, to the log-correlated Gaussian field $(X(x))_{x\in T}$. The later result of \cite{shamov2016gaussian} shows that two approximating sequences yield the same $M_{\gamma}$ almost surely.

Removing the restriction \cref{eq:sigmaType} was achieved in several steps. First, \cite{robert2010gaussian} defined $M_{\gamma}$ in the case that $R$ is a stationary positive definite kernel of the form \cref{eq:LogKernel} using a mollification approach, which in particular does not rely on a martingale approximation. Later, \cite{shamov2016gaussian} defined the GMC associated to general Gaussian fields using an approach based on random shifts. Finally, \cite{BerestyckiSimplePath} simplified the mollification approach and gave a simple proof that the limiting measure is independent of the chosen mollification procedure.
More on these approaches, as well as a rich theory on the behavior of GMC and its applications, can be found in \cite{RV14,berestycki-powell}.

There are various examples of log-correlated fields that are non-Gaussian in the pre-limit producing a limiting GMC  coming from random matrix theory; see, for instance, \cite{webb15,berestycki2018random,CN19,NSW20,claeys2021much,bf22,kivamae24,junnila2024multiplicative,lambert-najnudel}. 
In \cite{claeys2021much}, conditions on the exponential moments of the field under which one has a GMC limit are developed. These criteria and the fact that GMC is supported on its thick points have been used in \cite{lambert2020maximum,lambert2024law,peilen24} to determine the first order of the maximum of non-Gaussian log-correlated fields.

\subsection{Results on non-Gaussian multiplicative chaos}\label{Sec:BacknGMC}
While no general framework exists for non-Gaussian multiplicative chaos, we survey the relevant results in this subsection.

The simplest example of a non-Gaussian multiplicative chaos measure is given by the binary multiplicative cascade with non-Gaussian edge weights. This model produces a positive random measure $\mu$ on $[0,1]$ which can be characterized by the recursion
  \begin{equation}\label{eq:MultCascRec}
  \mu(\mathrm{d}t) \stackrel{d}{=} Z_0 \indset{[0,1/2]}(t)\mu_0(2\mathrm{d}t)+Z_1\indset{[1/2,1]}(t)\mu_1(2\mathrm{d}t-1) \,,
  \end{equation}
where $\mu_0$, $\mu_1$ are independent copies of $\mu$ and $Z_0$, $Z_1$ are nonnegative, non-Gaussian, i.i.d.\@ with mean 1. Motivated by \cref{eq:MultCascRec}, \cite{bacry2003log} studies multiplicative chaos measures for which the underlying field is given by the mass of certain cones under an infinitely divisible, independently scattered random measure. They show that, for their construction, a limiting measure exists and give a criterion for non-degeneracy and existence of moments. 
In \cite{10.1214/12-AOP810}, also motivated by 
\cref{eq:MultCascRec}, so-called $\star$-scale invariant random measures are studied, which, under mild  assumptions, are formally of the form $e^{\gamma L_x}\;\mathrm{d}x$ with $L$ a log-correlated and infinitely divisible process. They show the existence of $\star$-scale invariant random measures, characterize them (up to technical assumptions), and give a condition for non-degeneracy.

Brownian multiplicative chaos is the chaos measure one obtains by formally exponentiating the square root of the occupation field of 2D Brownian motion. This object was constructed in \cite{BBK94} in the $L^2$-regime and in \cite{AHS20,Jeg20a} in the full subcritical regime. Great progress has been made in its study and characterization since then \cite{loop-soup,jego21,jego23}. 
  
In \cite{fan2023trigonometric}, random trigonometric series of the form $\sum_n \alpha_n X_n\cos(nt+\Phi_n)$ are considered, where $\alpha_n \in \ell^4(\R)\setminus \ell^2(\R)$, $\Phi_n$ are i.i.d.\ with $\Phi_1\sim\mathrm{Unif}([0,2\pi))$, and $X_n$ are independent with $\mathbb{E}[X_n^2] = 1$ and $\mathbb{E}[X_n^{2m}] = O(K^m m!)$ for some $K>0$ and all $m\ge 2$, $n\ge 1$. They study  the associated chaos measures and give results on the Hausdorff dimension. It is worth mentioning that in the case $\alpha_n = \sqrt{2}n^{-1/2}$ and $X_n\sim\mathcal{N}(0,1)$ for all $n\in\N$, their model coincides with the random Fourier series \eqref{def:trig-field} with $a$ taken to be a Gaussian sequence, up to a scaling of space by $2\pi$. 

In \cite{SW20}, it is shown that a stochastic approximation of the logarithm of the Riemann zeta function may be decomposed into the sum of a Gaussian log-correlated field and a smooth field; using this, they show that the associated multiplicative chaos is absolutely continuous with respect to the GMC formed by the Gaussian part of the field with almost surely bounded Radon-Nikodym derivative.
We emphasize that in our case, in contrast to \cite{SW20}, the Radon-Nikodym derivative of $\mu_{\gamma,a}$ with respect to the $\mu_{\gamma,g}$ constructed in Theorem~\ref{Theo:Main} is not necessarily almost surely bounded.

A recent series of papers \cite{goro-wong-1,goro-wong-2,goro-wong-3}  constructs a non-Gaussian multiplicative chaos measure, in the  subcritical and critical regimes, associated to a Dirichlet series of coefficients coming from a certain class of random multiplicative functions restricted to the critical line. In fact, their results are robust under small perturbations away from the critical line. They do not attempt to obtain a coupling between their chaos measures and GMC measures. Relatedly, their methods do not involve Gaussian approximation.
  
Recently, \cite{2024arXiv240817219V} has shown that one can reconstruct the underlying log-correlated field from its multiplicative chaos measure (at any inverse temperature less than or equal to critical) for what they call ``mildly non-Gaussian'' fields, that is, fields of the form $X = G+H$, with $G$ a log-correlated Gaussian field and $H$ a H\"older continuous field.

The critical two-dimensional Stochastic Heat Flow (SHF) is a logarithmically-correlated process of random Borel measures on $\R^2$, constructed in \cite{CSZ23flow}, arising from directed polymers in random environment. In \cite{CSZ23gmc}, it is shown that the one-time marginal of the SHF is \textit{not} a GMC. There, the question is raised ``whether the critical 2d Stochastic Heat Flow is absolutely
continuous w.r.t.\ \textit{some} GMC.'' We do not address this question in the present article, however, after our paper was posted, a ``conditional GMC'' structure was uncovered in \cite{clark-tsai-25}.

Most relevant to this article is \cite{Junnila}, which, as mentioned above, proves for a quite general notion of log-correlated fields $(X_k(t))_{k\in\N, t\in [0,1]^d}$ on $[0,1]^d$ that, for any $K\subseteq U\subseteq \mathbb{C}$ with $K$ compact and $U\supset (0,\sqrt{2d})$ open, there is a $p>1$ for which the martingale 
  \[
  \mu_{n,\gamma}[f] := \int_{[0,1]^d} f(t) \frac{e^{\gamma\sum_{k=1}^n X_k(t)}}{\mathbb{E}[e^{\gamma\sum_{k=1}^n X_k(t)}]} \;\mathrm{d}t
  \] converges in $L^p(\Omega)$ for all $\gamma\in K$ and $f\in C([0,1]^d,\mathbb{C})$. To our knowledge, in this general case, independence of the limiting chaos measure $\mu_{\gamma}$ from the discrete approximation of $\sum_{k=1}^\infty X_k(t)$, the study of the multifractal spectrum, construction of the chaos at the critical $\gamma=\sqrt{2d}$, and more generally the question of which properties of GMC can be transfered to $\mu_{\gamma}$ are still open.
  
\subsection{Outline of the proof and structure of the paper}\label{subsec:proof-outline}
The proof of Theorem \ref{Theo:Main} begins with two major steps: a \emph{discretization step}, in which we show $\mu_{\gamma,a}$ is mutually absolutely continuous with the chaos associated to a discrete, hierarchical field (similar to a multiplicative cascade); and a \emph{coupling step}, in which we  construct a coupling between this field and a Gaussian field. 

The discretization step proceeds by considering a nested sequence of partitions of $[0,1]$ such that the $n^{\mathrm{th}}$ partition $\mathcal{P}_n$ contains approximately $2^n$ intervals (we will need slightly more, see \cref{eqn:size-Tn-Inj}). We then define the random series $\tS_{n,a}(t)$ in \cref{def:tree-field} which is a modification of the random Fourier series $S_{2^{n-1},a}(t)$. The main characteristics of $\tS_{n,a}(t)$ are that it (i) is composed of dyadically-growing chunks of $S_{2^{n-1},a}(t)$, so that the $\ell^{\mathrm{th}}$ increment of $\tS_{n,a}(t)$, $\ell \in \llbracket1,n\rrbracket$, is composed of a sum of $\approx 2^{\ell}$ terms of $S_{2^{n-1},a}$ and therefore has approximately constant variance $\log 2$ for all $\ell$; and (ii) is constant on each interval in $\mathcal{P}_n$. In light of (ii) and the fact that the $\mathcal{P}_n$ are nested, the field  $(\tS_{n,a}(t))_{t\in[0,1]}$ may be defined on an approximately binary tree; as such, the field resembles an inhomogeneous branching random walk (the major difference being that particles do not move independently after splitting). Indeed, this is the perspective we take, and we write $(\tS_{n,a}(v))_{v\in \cN_n} \equiv (\tS_{n,a}(t))_{t\in[0,1]}$, where $\cN_n$ denotes the set of particles in the $n^{\mathrm{th}}$ generation of the tree. 
Using the regularity of the eigenbasis $\{\cos(2\pi k t),\sin(2\pi k t)\}_{k\in\N}$, we show in \cref{Prop:TrigToHier} via a chaining argument that the chaos measure  $\widetilde{\mu}_{\gamma, \ta}$obtained from $(\tS_{n,a}(v))_{v\in\cN_n, n\in\N}$ by replacing $S$ by $\tS$ in \eqref{eq:PreLimGMC} is mutually absolutely continuous with respect to $\mu_{\gamma,a}$, almost surely.

The coupling step proceeds as follows. Note that by (i) above and the central limit theorem the $n^{\mathrm{th}}$ increment $\ta_{n}(v)$ of $\tS_{n,a}(v)$  becomes more and more Gaussian as $n$ increases. However, the increments $\{\ta_{n}(v)\}_{v\in \cN_{n}}$ are not independent from one another, and therefore we cannot couple them to Gaussian variables one-by-one. 
To circumvent this, we use the coupling result of Yurinskii (\cref{Lem:YrCoupSource} here, recorded from \cite{BELLONI20194} and originating from \cite{yurinskii1978error}) 
to directly couple a \emph{well-chosen subset} of $\{\tilde{a}_{n}(v)\}_{v\in\cN_n}$ as a vector to a Gaussian vector with the same covariance structure. A naive application of the Yurinskii coupling---i.e., trying to couple the entire vector of increments at level $n$ such that it is close in $L^{\infty}$-norm to a Gaussian vector---will fail because the dimension of this vector is too high. The critical insight to circumvent this issue is that we will only need to closely couple the subset of increments which are descendants of ``\emph{$\gamma$-thick points}'' of the field. Thus, we do not need to couple a vector of length $|\mathcal{P}_n| \approx 2^{n}$ but instead make do with a vector of length $\approx 2^{q_{\gamma}n}$, for some $q_{\gamma}$ that tends to $0$ as $\gamma \to \sqrt{2}$. This is accomplished in \cref{Lem:CouplThick}. Note the requirement $\gamma>1$ is introduced to ensure that the amount of random variables we need to closely couple is small enough and can be read explicitly from the right-hand side of \cref{eq:CouplingGoodOnThick}.

Ultimately, this leads to a coupled Gaussian field $\tS_g := (\tS_{n,g}(v))_{v\in\cN_n, n\in\N}$ constructed on the same tree as $(\tS_{n,a}(v))_{v\in\cN_n, n\in\N}$. \cref{Prop:HierGaussAppr} states that their associated chaos measures are mutually absolutely continuous. Up to and including this result, we have shown that $\mu_{\gamma, a}$ is absolutely continuous with respect to the GMC associated to the discrete, hierarchical Gaussian model $\tS$.
The proof of \cref{Theo:Main}, given in \cref{subsec:pf-main-theorem}, follows because the construction of the hierarchical model is in a sense ``invertible'': an associated Fourier series with i.i.d.\ $\cN(0,1)$ coefficients can be constructed from $\tS_g$.

The proof of \cref{Prop:HierGaussAppr} occupies the bulk of the paper. We construct the Radon-Nikodym derivative between these two chaos measures by studying integrability properties of the obvious candidate, defined in \eqref{def:RN-derivative}, with respect to each chaos measure. Our proof proceeds via a thick-point analysis and moment estimates using barriers. We refer to the start of \cref{sec:pf-HierGaussAppr} for a more technical outline.

In \cref{sec:discretization+coupling}, we carry out the discretization and coupling steps. In \cref{sec:pf-HierGaussAppr}, we prove \cref{Prop:HierGaussAppr}. In \cref{sec:generalizations-future-directions}, we discuss open questions raised by our work, as well as the extension to general bounded domains (see \cref{thm:general}).

\subsection{Some notation}
Below, we use $c,C>0$ to denote constants depending only on the distribution of $a_1^{(1)}$ that may change from line to line. We also say $f(n) = O_n(g(n))$ if $|f(n)| \leq Cg(n)$ for all $n\geq 1$,  and $f(n) = o_n(g(n))$ if $f/g  \to 0$ as $n\to\infty$. The set of positive integers is denoted by $\N$. For $m\le n\in\N$ we set $\llbracket m, n \rrbracket := \{m,\dots, n\}$.

\section{Discretization, Gaussian coupling, and proof of \texorpdfstring{\cref{Theo:Main}}{Theorem 1.1}}
\label{sec:discretization+coupling}
In \cref{subsec:hierarchical-model}, we define a discretization of the field \eqref{def:trig-field} on a hierarchical lattice. We then state and prove \cref{Prop:TrigToHier}, which concerns mutual absolute continuity of $\mu_{\gamma,a}$ with the  chaos measure associated to this discrete hierarchical model. In \cref{subsec:Gaussian-models}, we construct a coupling between the hierarchical model and a Gaussian hierarchical model using the Yurinskii coupling \cite{yurinskii1978error} in a version for the supremum norm proved in \cite{BELLONI20194}.
We then state \cref{Prop:HierGaussAppr}, which yields absolute continuity between the associated chaos measures. 
In \cref{subsec:pf-main-theorem}, we prove \cref{Theo:Main}.

\subsection{The discrete hierarchical model}
\label{subsec:hierarchical-model}
We begin by defining a sequence of partitions $\{\mathcal{P}_n\}_{n\geq 1}$ of $[0,1)$ in the following manner. Let $f:\N\to\R, n\mapsto n^4$ and $\mathcal{S}_1 := \{0,1\}$. Given $\mathcal{S}_n$ we define 
\[
    \mathcal{S}_{n+1} := \mathcal{S}_n\cup \{\tfrac{j}{f(n)2^n}\}_{j\in \llbracket 1, 2^nf(n) \rrbracket}.
\]
Let $(\tau_n(j))_{j\in \llbracket 1, |\mathcal{S}_n| \rrbracket}$ be the monotonically increasing enumerate of $\mathcal{S}_n$. Define $T_n := |\mathcal{S}_n|-1$,  $I_n(j) := [\tau_n(j), \tau_n(j+1))$ for $j\in \llbracket1,T_n\rrbracket$, and define
\[
  \mathcal{P}_n := \{I_n(j)\}_{j\in \llbracket1,T_n\rrbracket}\,.
\]
We emphasize that the points $\tau_n(j)$ are not equally spaced, and so the intervals $I_n(j)$ do not necessarily have the same length. At this point, it is useful to record
\begin{align}\label{eqn:size-Tn-Inj}
  T_n \leq f(n) 2^{n+1} \qquad \text{and} \qquad |I_n(j)| \leq 2^{-n}f(n)^{-1}\,.
\end{align}
For each $n\geq 1$ and $j\in \llbracket 1, T_n\rrbracket$, we choose a representative element $t_n(j)$ of $I_n(j)$ in such a way that every $2^{n-1}\times 2^{n-1}$ square submatrix of both
\begin{multline}\label{def:Cn}
  C_n^{(1)} := \Big(\frac{1}{\sqrt{k}} \cos(2\pi k t_n(j)) \Big)_{k \in \llbracket 2^{n-1}, 2^{n}-1\rrbracket ,  j \in \llbracket 1, T_n\rrbracket} \,, \qquad \text{and}\\  \qquad C_n^{(2)} := \Big(\frac{1}{\sqrt{k}} \sin(2\pi k t_n(j)) \Big)_{k \in \llbracket 2^{n-1}, 2^{n}-1\rrbracket , j \in \llbracket 1, T_n\rrbracket}
\end{multline}
is invertible. This is possible because the   determinants of the $2^{n-1}\times 2^{n-1}$ square submatrices of $C_n^{(1)}$ and $C_n^{(2)}$ are analytic functions of $(t_n(j))_{j\le T_n}$;  as such they have a zero set of Lebesgue measure $0$, which implies that we can choose the $t_n(j)$ such that all these determinants are non-zero. The exact choice of $t_n(j)$ is of no importance. 
This invertibility condition allows us to construct a Fourier series with i.i.d.\ random coefficients from a discrete hierarchical model, which leads to the construction of the $g_k^{(i)}$ appearing in \cref{Theo:Main}, see \eqref{def:iid-gaussians}.

Note that $\mathcal{P}_{n+1}$ is a refinement of $\mathcal{P}_n$ for all $n\geq 1$, and as such, the system of intervals $\{I_n(j) : j\in \llbracket 1,T_n \rrbracket\}$ (as well as their representative elements $t_n(j)$) naturally inherits a tree structure. In particular, we construct a tree $\mathfrak{T}$ satisfying, (1) the vertices in the $n^{\mathrm{th}}$ level of the tree are in bijection with  the representatives $\{t_n(j): j\in \llbracket 1, T_n \rrbracket \}$, and (2) the descendants of the vertex corresponding to the interval $I_n(j)$ are the nodes in the $(n+1)^{\mathrm{th}}$ level of the tree corresponding to the intervals in $\{I_{n+1}(j): j\in \llbracket 1, T_{n+1}\rrbracket\}$ contained in $I_n(j)$. 

The following notation makes  this bijection explicit, translating between the tree $\mathfrak{T}$ and the system of representative elements. We let $\cN_n$ denote the collection of vertices in the $n^{\mathrm{th}}$ level of $\mathfrak{T}$, so that $|\cN_n| = T_n$ and $V(\mathfrak{T}) = \cup_{n\geq 1} \cN_n$. 
For $v \in \cN_n$ corresponding to $t_n(j)$, define 
\[
  t_n(v):= t_n(j) \qquad \text{and} \qquad I_n(v) := I_n(j)\,.
\]
For $1\leq m < n$, $v \in  \cN_n$, and $w\in \cN_m$, write 
$w \prec v$ if $v$ is a descendant of $w$, and define 
\[
  t_m(v):= t_m(w) 
\]
i.e., $t_m(v)$ is the representative of the generation-$m$ parent interval corresponding to $v$.
For each $n\geq 1$, let 
\[
  \pi_n: [0,1] \to \{t_n(j)\}_{j\leq T_n}
\]
denote the map that sends $t\in I_n(j)$ to $t_n(j)$ and sends $1$ to $1$.
For each $n\geq 1$, 
let 
\[
  v_n: [0,1] \to \cN_n
\]
denote the map that sends $t \in I_n(j)$ to the vertex $v_n(t)$ in $\cN_n$ corresponding to $t_n(j)$.

Having embedded a tree structure into $[0,1]$, we now consider a hierarchical discretization of the field $(S_{n,a}(t))_{t\in [0,1]}$ by evaluating the field  on this tree.
We define the following quantities for $n\geq 1$, $k\in\llbracket 2^{n-1}, 2^n-1\rrbracket$, $t\in[0,1]$, and $v \in \cN_n$:
\begin{equation}\label{Def:Coeffs}
\begin{aligned}
\Co{k}{n}{v} &:= \frac{1}{\sqrt{k}}\cos(2\pi k t_n(v)) 
\quad \text{and} \quad
\Ct{k}{n}{v} := \frac{1}{\sqrt{k}}\sin(2\pi k t_n(v))\,,\\
 \ta_n^{(i)}(v) &:= (\Ci{k}{n}{v})_{k\in \{2^{n-1},\dots, 2^{n}-1\}} \cdot (\ai_k)_{k\in \{2^{n-1},\dots, 2^{n}-1\}}\,, \text{ for } i \in \{1,2\}\,, \text{ and} \\
 \ta_n &:= \ta_{n}^{(1)}+\ta_{n}^{(2)}.
\end{aligned}
\end{equation}
Observe that our choice of the $t_n(v)$ ensures that any $2^{n-1}\times2^{n-1}$-submatrix of $(C_{k,n}^{(i)}(v))_{k,v}$ is invertible (see below \cref{def:Cn}).
Observe further that $\ta_n$ is a vector in $\R^{T_n}$, recalling $T_n= |\cN_n|$. Denote its coordinates by $\ta_n := (\ta_{n}(v))_{v\in\cN_n}$, and note \eqref{Def:Coeffs} gives
\begin{align}\label{def:tree-field-increments}
  \ta_n(v) = \sum_{k=2^{n-1}}^{2^{n}-1} \frac{1}{\sqrt{k}}\left(\ao_k\cos(2\pi k t_n(v)) +\at_k\sin(2\pi k t_n(v))\right) = \sum_{k=2^{n-1}}^{2^{n}-1} X_k(t_n(v))\,,
\end{align}
where, for ease of notation, we define
\begin{equation}\label{eq:SummandsField}
X_k(t) := k^{-1/2} \left(\ao_k\cos(2\pi k t)+\at_k\sin(2\pi kt) \right) \qquad \text{for } t\in [0,1]\,.
\end{equation}
Though we prefer to think of our model as being defined on a tree, it will occasionally be more convenient for calculations to define
 \[
 \Ci{k}{n}{t}:= \Ci{k}{n}{v_n(t)} \qquad \text{and} \qquad  \ta_n(t) := \ta_n(v_n(t))\quad\text{for}\ t\in [0,1]\,.
 \]
 For $1\leq m <n$, $v \in \cN_n$, and $w \in \cN_m$ such that $w \prec v$,  define $\ta_m(v) := \ta_m(w)$. We then associate a random walk to each $v \in \cN_n$ (and therefore, each $t\in [0,1]$): 
 \begin{equation}\label{def:tree-field}
    \tS_{k,a}(v):= \sum_{m=1}^k \ta_m(v) \quad \text{and} \quad \tS_{k,a}(t) := \sum_{m=1}^k \ta_m(v_m(t)) \,,\quad \text{ for } 1\leq k \leq n\,.
 \end{equation}
Observe that if $t$ lies in the interval corresponding to $v\in \cN_n$, then $\tS_{k,a}(v) = \tS_{k,a}(t)$ for all $1\leq k \leq n$. Indeed, we have 
\begin{equation}\label{eqn:equivalence-of-fields}
  S_{2^n-1,a}(t_n(j)) = \tS_{n,a}(t) \,, \quad \text{for all $n\geq 1$, $j\leq T_n$, and $t\in I_n(j)$.}
\end{equation}

Finally, we are ready to define the associated chaos measure. For $n\in\N$, $\gamma\in\R$, $(a^{(i)}_k)_{k\in\N, i\in\{1,2\}}$ i.i.d.\@ satisfying \cref{eq:CondsOnak}, and $A \subset [0,1]$ Borel, we set
\begin{equation}\label{eq:PreLimHierGMC}
\widetilde{\mu}_{n,\gamma,\ta}[A] := \int_A \tZ_{n,\gamma,\ta}(t)^{-1}e^{\gamma \tS_{n,a}(v_n(t))}\;\mathrm{d}t \,,
\end{equation}
where $\tZ_{n,\gamma,\ta}(t) := \mathbb{E}[e^{\gamma \tS_{n,a}(v_n(t))}]$. In Remark \ref{rk:junnila-estimates}, we show that $\tS_{n,\ta}(t)$ is a log-correlated field so that \cite[Theorem~2]{Junnila} implies $\widetilde{\mu}_{n,\gamma,\ta}$ converges almost surely in the space of Radon measures on $[0,1]$ equipped with the topology of weak convergence to a random measure $\widetilde{\mu}_{\gamma,\ta}$, which we call the \textit{hierarchical chaos measure} corresponding to $(a_{k}^{(i)})_{k\in\N,i\in\{1,2\}}$.

\begin{remark}\label{rk:junnila-estimates}
The fields $(S_{n,a}(t))_{t\in[0,1]}$ and $(\tS_{n,a}(t))_{t\in[0,1]}$, defined in \cref{def:trig-field,def:tree-field} respectively, are \textit{log-correlated fields}. In particular, their increments $X_k(t)$ fulfill the conditions in \cite[Definition~4]{Junnila}. Indeed, that $(S_{n,a}(t))_{t\in[0,1]}$ satisfies these conditions is shown in \cite[Section~3.1]{Junnila}, and using the identity \cref{eqn:equivalence-of-fields}, it follows immediately that $(\tS_{n,a}(t))_{t\in[0,1]}$ does as well.
Therefore, we may apply the estimates developed in \cite{Junnila} for such fields to $S_{n,a}$ and $\tS_{n,a}$.
\end{remark}

 We repeatedly use the following result on the Laplace transforms of $S_{n,a}(t)$, $\widetilde{S}_{n,\widetilde{a}}(t)$:

\begin{lemma}\label{lem:junnila-lemma-7}
Fix $R>0$. For $\gamma_1, \gamma_2\in [-R,R]$  and $s,t\in [0,1]$,
\begin{align*}
&\begin{multlined}[t]\mathbb{E}[e^{\gamma_1S_{2^n-1,a}(t)+\gamma_2 S_{2^n-1,a}(s)}] \\
= e^{\frac{\gamma_1^2}{2}\Var[S_{2^n-1,a}(t)]+\frac{\gamma_2^2}{2}\Var[S_{2^n-1,a}(s)]+\gamma_1\gamma_2\Cov[S_{2^n-1,a}(t),S_{2^n-1,a}(s)]+\zeta(\gamma_1,\gamma_2,t,s)+O_n(2^{-n/2})},\end{multlined}\\
&\text{and}\\
&\begin{multlined}[t]\mathbb{E}[e^{\gamma_1\widetilde{S}_{n,\widetilde{a}}(t)+\gamma_2 \widetilde{S}_{n,\widetilde{a}}(s)}] \\
= e^{\frac{\gamma_1^2}{2}\Var[\widetilde{S}_{n,\widetilde{a}}(t)]+\frac{\gamma_2^2}{2}\Var[\widetilde{S}_{n,\widetilde{a}}(s)]+\gamma_1\gamma_2\Cov[\widetilde{S}_{n,\widetilde{a}}(t),\widetilde{S}_{n,\widetilde{a}}(s)]+\widetilde{\zeta}(\gamma_1,\gamma_2,t,s)+O_n(2^{-n/2})},\end{multlined}
\end{align*}
where the implicit constants in the $O_n$ terms as well as $\max\{|\zeta(\gamma_1,\gamma_2,t,s)|,|\widetilde{\zeta}(\widetilde{\gamma_1},\widetilde{\gamma_2},\widetilde{t},\widetilde{s})|\}$ are all bounded by a constant $C_R>0$ depending only on $R$ and the distribution of $a$.

\begin{proof}
The result essentially comes from examining the proof of \cite[Lemma~7]{Junnila}. We shall focus on proving the second equality above, as the first equality follows similarly.

As written, \cite[Lemma~7]{Junnila}\footnote{For ease of reading, let us translate the notation appearing in \cite[Lemma~7]{Junnila} and its proof to our notation. Take their $X_k$ to be our $X_k$, and take $j= n$, $t_n = 2^n$, $x = \pi_m(t)$, $y = \pi_m(s)$, $\xi_1 = \gamma_1$, and $\xi_2 =\gamma_2$.} collects the desired $\widetilde{\zeta}(\gamma_1, \gamma_2,t,s)+O_n(2^{-n/2})$ term 
 into a term simply denoted by $\epsilon$, which is shown to be bounded by a constant $C_R$ uniformly over all $n, t,s, \gamma_1$, and $\gamma_2$. Therefore, it remains to extract this quantitative decomposition From $\epsilon$.
The key estimate for this is the following: for any $p,q >0$, there exists a constant $\mathcal{E}(p+q)>0$ such that for all $t,s, \in [0,1]$ and $k \in \N$, 
\begin{align*}
  \E\Big[|X_k(t)|^p |X_k(s)|^q \Big] \leq \mathcal{E}(p+q) k^{-\frac{p+q}{2}}\,.
\end{align*}
This is immediate from \eqref{eq:CondsOnak} and \eqref{eq:SummandsField}.

Now, as in the first line of the proof of \cite[Lemma~7]{Junnila}, define 
\[
  \varphi_k(\gamma_1, \gamma_2):=  \varphi_k(\gamma_1, \gamma_2, t,s) = \E [e^{\gamma_1 X_k(\pi_m(t)) + \gamma_2 X_k(\pi_m(s))} ]
\]
where $m= \lfloor\log_2 k\rfloor+1$ and  we recall $\pi_m$ from \eqref{subsec:hierarchical-model} (we will not use the precise definition of $\pi_m$ as all bounds will be given uniformly over $t$ and $s$). We then follow the proof of \cite[Lemma~7]{Junnila} until the second-to-last display,
where we instead define 
\begin{align*}
  \varphi(\gamma_1, \gamma_2) := \varphi(\gamma_1, \gamma_2, t,s)  = \mathbb{E}[e^{\gamma_1\widetilde{S}_{n,\widetilde{a}}(t)+\gamma_2 \widetilde{S}_{n,\widetilde{a}}(s)}]
  = \prod_{m=1}^n \prod_{k=2^{m-1}}^{2^m-1} \varphi_k(\gamma_1, \gamma_2, t,s)\,.
\end{align*}
Recall $k_0$ from the first line of \cite[Page~6176]{Junnila}, which was shown to depend only on $R$. Define $m_0 := \inf \{m \in \mathbb{N} : 2^m -1 \geq k_0\}$, and write 
\[
  \varphi(\gamma_1, \gamma_2) = 
  \left(\prod_{m=1}^{m_0}  \prod_{k=2^{m-1}}^{2^m-1} \varphi_k(\gamma_1, \gamma_2,t,s) \right) 
  \left(\prod_{m=m_0+1}^{n} \prod_{k=2^{m-1}}^{2^m-1} \varphi_k(\gamma_1, \gamma_2,t,s) \right) 
\]
(this takes the place of the second-to-last display of the proof of \cite[Lemma~7]{Junnila}).
Denote the logarithm of the first product by 
\[
  \widetilde{\zeta}_1(\gamma_1,\gamma_2,t,s) := \log \left(\prod_{m=1}^{m_0}  \prod_{k=2^{m-1}}^{2^m-1} \varphi_k(\gamma_1, \gamma_2,t,s) \right)\,,
\] 
which for all $\gamma_1, \gamma_2, t$ and $s$ is bounded by a constant depending only on $R$ (due to the same being true for $m_0$ and \eqref{eq:CondsOnak}) 
and will eventually be absorbed into $\widetilde{\zeta}(\gamma_1,\gamma_2,t,s)$. Similar to the final display of \cite[Lemma~7]{Junnila}, the logarithm of the second product is
\begin{align*}
  \frac{\gamma_1^2}{2}\Var[\widetilde{S}_{n,\widetilde{a}}(t)]+\frac{\gamma_2^2}{2}\Var[\widetilde{S}_{n,\widetilde{a}}(s)]+\gamma_1\gamma_2\Cov[\widetilde{S}_{n,\widetilde{a}}(t),\widetilde{S}_{n,\widetilde{a}}(s)] \\
  +\widetilde{\zeta}_2(\gamma_1, \gamma_2,t,s)
  - \sum_{k=2^n}^{\infty} d_k O(\gamma^3+ \gamma^6)\,,
\end{align*}
where $\gamma := \max(|\gamma_1|, |\gamma_2|)$, the $O$ term is independent of $t$ and $s$, and 
\begin{align*}
  \widetilde{\zeta}_2(\gamma_1, \gamma_2,t,s)& \\
  := &-\frac{\gamma_1^2}{2}\sum_{k=1}^{2^{m_0}-1} \E[X_k(\pi_m(t))^2] 
  -\frac{\gamma_2^2}{2}\sum_{k=1}^{2^{m_0}-1} \E[X_k(\pi_m(s))^2] \\
  &-\gamma_1\gamma_2\sum_{k=1}^{2^{m_0}-1} \E[X_k(\pi_m(t))X_k(\pi_m(s))] + \sum_{k=2^{m_0}-1}^{\infty} d_k O(|\gamma|^3+ |\gamma|^6)\,.
\end{align*}
The equations leading up to \cite[Eq.(14)]{Junnila} and our moment bound on $\E[|X_k(t)|^p |X_k(s)|^q]$ in the first display of the proof show  that $|c_k| = O(k^{-3/2})$. Our moment bound and the definition of $d_k$ in \cite{Junnila} imply $|d_k| = (1+o(1))|c_k| = O(k^{-3/2})$. Therefore, we find $|\widetilde{\zeta}_2(\gamma_1,\gamma_2,t,s)|$ is bounded by a constant depending only on $R$, and 
\[
  \sum_{k=2^n}^{\infty} d_k O(\gamma^3+\gamma^6) = O(2^{-n/2})\,.
\]
Taking $\widetilde{\zeta} = \widetilde{\zeta}_1+ \widetilde{\zeta}_2$ concludes the proof.
\end{proof}
\end{lemma}
As an immediate consequence of Lemma \ref{lem:junnila-lemma-7} and the fact that both $\Var[\widetilde{S}_{n,\widetilde{a}}(t)]$, $\Var[S_{2^n-1,a}(t)]$ may be expressed as $n\log(2)+o_n(1)$ we get the following estimate on the sequences of normalizing functions in \cref{eq:PreLimGMC,eq:PreLimHierGMC}:  for each $t\in [0,1]$ and $\gamma \in (0,\sqrt2)$, there exist $\zeta(\gamma,t)$ and $\widetilde{\zeta}(\gamma,t)$ such that, for all $n\in \N$,
\begin{multline}\label{eqn:junnila-lemma-7}
  Z_{2^n-1,\gamma, a}(t) = \exp\Big(\frac{\gamma^2}2\log(2)n +\zeta(\gamma,t) +o_n(1) \Big) \\
  \text{and} \quad 
  \tZ_{n,\gamma, \ta}(t) = \exp\Big(\frac{\gamma^2}2\log(2)n +\widetilde{\zeta}(\gamma,t) +o_n(1) \Big)\,,
\end{multline}
where the $o_n(1)$ terms have no $t$ dependence since the $O_n$ terms in Lemma~\ref{lem:junnila-lemma-7} do not.

\cref{Prop:TrigToHier} below justifies the introduction of the hierarchical chaos measure. We prove the proposition via a chaining argument.

\begin{proposition}\label{Prop:TrigToHier}
Let $\gamma\in(0,\sqrt{2})$, fix $(a_k^{(i)})_{k\in\N}$ i.i.d.\@ with \cref{eq:CondsOnak} and let $(\ta_n)_{n\in\N}$ be defined as in \eqref{Def:Coeffs}. Then, almost surely,
\[
\widetilde{\mu}_{\gamma, \ta}\ll \mu_{\gamma,a} \ll \widetilde{\mu}_{\gamma,\ta}\,.
\]
\end{proposition}
\begin{proof}
We use $(\Omega, \mathcal{F},\mathbb{P})$ to denote the probability space on which $(a_k^{(i)})_{k\in\N,i\in\{1,2\}}$ is defined. It suffices to show that there exists an almost surely finite random variable $\mathfrak{C}(\omega)$ such that for all $A\subset [0,1]$ Borel, we have
\[
\mathfrak{C}(\omega)^{-1}\widetilde{\mu}_{\gamma,\ta}[A] \le \mu_{\gamma,a}[A]\le \mathfrak{C}(\omega)\widetilde{\mu}_{\gamma,\ta}[A].
\]
For this we set 
\[
R_{n,\gamma}(t) := \tZ_{n,\gamma,\ta}(t)^{-1}Z_{2^n-1,\gamma,a}(t) e^{\gamma \tS_{n,a}(v_n(t)) -\gamma S_{2^n-1,a}(t)}.
\]
By the definitions in \cref{eq:PreLimGMC}, \cref{eq:PreLimHierGMC} we have that for all $A\subseteq [0,1]$ Borel, $\mathbb{P}$-a.s.,
\begin{align*}
\mu_{\gamma,a}[A] &= \lim\limits_{n\to \infty} \int_A Z_{2^n-1,\gamma}(t)^{-1} e^{\gamma S_{2^n-1,a}(t)} \d t\,,\quad \text{and}\\
  \widetilde{\mu}_{\gamma, \ta}[A] &= \lim\limits_{n\to\infty} \int_A \tZ_{n,\gamma,\ta}(t)^{-1} e^{\gamma \tS_{n,a}(v_n(t))} \d t = \lim\limits_{n\to\infty}\int_A R_{n,\gamma}(t) Z_{2^n-1,\gamma}(t)^{-1} e^{\gamma S_{2^n-1,a}(t)}\;\mathrm{d}t,
\end{align*}
so that to prove Proposition \ref{Prop:TrigToHier} it suffices to show that $ \sup_{n\in\N}\sup_{t\in[0,1]} |\log( R_{n,\gamma}(t))|$ is almost surely finite. From \cref{eqn:junnila-lemma-7}, we know that $\tZ_{n,\gamma,\ta}(t)^{-1}Z_{2^n-1,\gamma,a}(t)$ converges uniformly in $t$.
Thus, it suffices to show that, almost surely,
\begin{equation}\label{eq:SupBounded}
\sup_{n\in\N}\sup_{t\in[0,1]} | \tS_{n,a}(t) - S_{2^n-1,a}(t) |<\infty.
\end{equation}
Recalling the map $\pi_n$ from \cref{subsec:hierarchical-model}, along with \eqref{Def:Coeffs} and \cref{eq:SummandsField}, we can write
\begin{align*}
  | \tS_{n,a}(t) - S_{2^n-1,a}(t) |
  = \Bigg|\sum_{m=1}^n \sum_{k=2^m-1}^{2^m-1} X_k(t)-X_k(\pi_m(t)) \Bigg|\,.
\end{align*}
In light of Borel-Cantelli, to show \cref{eq:SupBounded}, it suffices to show the existence of $C_0>0$ such that
\begin{equation}\label{eq:GoalSupBound}
\sum_{m\in\N} \mathbb{P}\Bigg( \sup_{t\in[0,1]} \Bigg|\sum_{k=2^m+1}^{2^{m+1}} X_k(t)-X_k(\pi_m(t))\Bigg| > C_0m^{-2}\Bigg)<\infty. 
\end{equation}
We do this via a chaining argument
for the continuous random process $(Y_m(t))_{t\in[0,1]}$ given by
\begin{equation}\label{eq:DefYmt}
Y_m(t) := \sum_{k=2^m+1}^{2^{m+1}} X_k(t)-X_k(\pi_m(t)).
\end{equation}
Observe \cref{eqn:equivalence-of-fields} implies $Y_m(t_m(j)) = 0$ for all $j \in \llbracket 1,T_m \rrbracket$, and so the chaining argument proceeds by considering a dyadic mesh of points within each interval $I_m(j)$, enumerated as follows. For $\ell,m\in\N$, define 
\[
\widetilde{\mathcal{P}}_{\ell,m} := \bigcup_{j=1}^{T_m} \bigcup_{i=1}^{2^\ell} \{t_m(j)+i2^{-\ell} |I_m(j)|\},
\] 
so that $|\widetilde{\mathcal{P}}_{\ell,m}| = T_m 2^\ell$. Furthermore, let $\rho_{\ell,m}(t)$ denote the unique element $r\in \widetilde{\mathcal{P}}_{\ell,m}$ minimizing distance to $t$ and satisfying $\pi_m(t) = \pi_m(r)$ (so that $t$ and $\rho_{\ell,m}(t)$ lie in the same interval $I_m(j)$); in particular, $\rho_{0,m}(t) = \pi_m(t)$ for each $t\in[0,1]$. Since $Y_m(t)$ is continuous on $[0,1]\setminus \mathcal{P}_m$ and equal to $0$ on $\mathcal{P}_m$, we have the following for all $t\in [0,1]$:
\begin{align}\label{eq:DecompIntoIncrs}
Y_m(t) = \sum_{\ell\in \N} Y_m(\rho_{\ell,m}(t))-Y_m(\rho_{\ell-1,m}(t))\,.
\end{align}

We now record the key estimate. Consider $\ell, m\in \N$ and $s,t \in \widetilde{\mathcal{P}}_{\ell,m}$ that are consecutive, so that $|s-t| = 2^{-\ell} |I_m(j)|$ for some $j \leq T_m$ (and therefore $\pi_m(t) =\pi_m(s)$). Then, using the Lipschitz property of cosine/sine as well as \cref{eqn:size-Tn-Inj}, we find 
\begin{multline*}
  \P \bigg( |Y_m(t)-Y_m(s)| > 2\sqrt2 \pi m^{-2} \ell^{-2} \bigg)\\
 = \P \bigg( \Big| \sum_{k=2^{m-1}}^{2^m-1} X_k(t) - X_k(s) \Big| > 2\sqrt2 \pi m^{-2} \ell^{-2} \bigg)
  \leq Ce^{- c\frac{f(m) 2^{\ell}}{m^2 \ell^2}},
\end{multline*}
where we have used the exponentially-decaying right and left tails of the sum given to us by \cite[Chapter III, Theorem~16]{petrov} (in particular, Eq.~(4.3) and (4.6) there).\footnote{Note that \cite[Chapter III, Lemma~5]{petrov} allows us to apply Theorem~16 of that chapter, as condition (4.1) is satisfied.}
For each $m\in \N$, a union bound over $\ell \in \N$ and all possible pairs of $(\rho_{\ell-1,m}(t), \rho_{\ell,m}(t))$ then yields  
\begin{align*}
  \P\Big(\exists_{\ell\in\N}\exists_{t\in[0,1]} :  |Y_m(\rho_{\ell,m}(t))-Y_{m}(\rho_{\ell-1,m}(t))| > 2\sqrt2\pi m^{-2}\ell^{-2}\Big) \le C\sum_{\ell\in \N} T_m 2^{\ell} e^{- c\frac{f(m) 2^{\ell}}{m^2 \ell^2}}. 
\end{align*}
Combining the last display with \cref{eq:DecompIntoIncrs} yields 
\begin{align*}
\mathbb{P}\Big(\sup_{t\in [0,1]} |Y_m(t)|> \sqrt2\pi^3/3\cdot m^{-2}\Big) \le C\sum_{\ell\in \N} T_m 2^{\ell+1} e^{- c\frac{f(m) 2^{\ell}}{m^2 \ell^2}}\,,
\end{align*}
which, recalling the definition of $Y_m(t)$ in \cref{eq:DefYmt} as well as the bound on $T_m$ from \cref{eqn:size-Tn-Inj}, implies \cref{eq:GoalSupBound}.
This concludes the proof of \cref{Prop:TrigToHier}.
\end{proof}

\subsection{Coupling with Gaussian log-correlated fields}\label{subsec:Gaussian-models}
In addition to \cref{Prop:TrigToHier}, the other key input for the proof of \cref{Theo:Main} is 
\cref{Prop:HierGaussAppr}, which gives mutual absolute continuity of $\widetilde{\mu}_{\gamma, \ta}$ with a chaos measure associated to a discrete hierarchical Gaussian model constructed on the same probability space.
Before stating the proposition, we construct a coupling between $(\ta_n^{(i)})_{i \in \{1,2\},n\in\N}$ and a Gaussian random vector $(\tg_n^{(i)})_{i\in\{1,2\},n\in\N}$, where $\tg_n^{(i)} := (\tg_n^{(i)}(v))_{v\in\cN_n}$, such that the two vectors are close on coordinates that are  ``thick'' (\cref{Lem:CouplThick}), and 
\begin{align}\label{eqn:tg-ta-sameCov}
  \Cov[(\ta_n^{(1)},\ta_n^{(2)})_{n\in\N}] = \Cov[(\tg_n^{(1)},\tg_n^{(2)})_{n\in\N}]\,.
\end{align}
Note \cref{eqn:tg-ta-sameCov} implies ``generational independence'': $\tg_n^{(i)}$ is independent of $\tg_m^{(i)}$ for $n\neq m$.
In fact, we construct i.i.d.\ standard Gaussians $g_k^{(i)}$ such that the $\tg_n^{(i)}$ are given by taking $a_k^{(i)} := g_k^{(i)}$ in \eqref{Def:Coeffs}. These will be the $g_k^{(i)}$ appearing in the statement of \cref{Theo:Main}. The construction is completed in \eqref{def:tilde-gaussians}.
The main tool here is a variant of the Yurinskii coupling from
\cite{BELLONI20194}, reproduced below. 
\begin{lemma}[{\cite[Appendix~I, Lemma~38]{BELLONI20194}}]\label{Lem:YrCoupSource}
For any $n, K \in \N$, 
let $\xi_1,\dots, \xi_n$ be independent centered random vectors in $\R^K$ such that
\begin{equation}\label{eq:Beta}
  \beta := \sum_{k=1}^n \mathbb{E}\left[\|\xi_k\|_2^2 \, \|\xi_k\|_{\infty} \right]+\mathbb{E}\left[\|\mathfrak{g}_k\|_2^2 \, \|\mathfrak{g}_k\|_{\infty} \right]
\end{equation}
is finite, where each $\mathfrak{g}_k$ denotes an independent random vector in $\R^K$, $\mathfrak{g}_k\sim \mathcal{N}(0,\Cov(\xi_i))$. Define $S := \sum_{k=1}^n \xi_k$ and let $Z\sim \mathcal{N}(0,I_K)$. Then, for each $\delta>0$, there exists a random vector $T\sim \mathcal{N}(0,\Cov(S))$ such that
\[
  \mathbb{P}(\|S-T\|_{\infty}> 3\delta) \le \min_{t\ge 0}\Big\{2\mathbb{P}(\|Z\|_{\infty} >t)+\frac{\beta}{\delta^3}t^2\Big\}\,.
\] 
\end{lemma}
As a corollary of the last lemma we get the following statement.

\begin{lemma}\label{Lem:YurCoupl}
Fix $n\in\N$ and $\mathcal{K}\subseteq \mathcal{N}_n$. There exists a $\mathcal{N}(0,\Cov[(\ta_{n}^{(1)}(v),\ta_{n}^{(2)}(v))_{v\in\mathcal{K}}])$-distributed random vector $\mathfrak{g}_{\mathcal{K},n}$ taking values in $\R^{2|\mathcal{K}|}$ such that
\begin{equation}\label{eq:CouplingQuality}
\mathbb{P}\Big(\big\| \mathfrak{g}_{\mathcal{K},n}-(\ta_{n}^{(1)}(v),\ta_{n}^{(2)}(v))_{v\in\mathcal{K}} \big\|_{\infty}>n^{-2} \Big ) \leq Cn^82^{-n/2}|\mathcal{K}| \log(|\mathcal{K}|)^{3/2} \,.
\end{equation}
\end{lemma}
\begin{proof}
For notational brevity, we  define $N = 2^{n-1}$ and $l_n(k) = 2^{n-1}-1+k$ for $k\in \llbracket 1, N\rrbracket $. 
We seek to apply Lemma~\ref{Lem:YrCoupSource} with $\delta:= n^{-2}/3$, $K := 2|\mathcal{K}|$,  and
\[
   \xi_k := \left(\frac{a_{l_n(k)}^{(1)}}{\sqrt{l_n(k)}}\cos(2\pi  l_n(k) t_n(v)) \,,\, \frac{a_{l_n(k)}^{(2)}}{\sqrt{l_n(k)}}\sin(2\pi l_n(k) t_n(v))\right)_{v\in \mathcal{K}} \,,
\]
so that $\sum_{k\in \llbracket 1, N\rrbracket } \xi_k = (\ta_n^{(1)}(v),\ta_n^{(2)}(v))_{v\in\mathcal{K}}$.
Further, since $(a_k^{(i)})_{i\in\{1,2\},k\in\N}$ are i.i.d.\ with mean zero, the sequence $(\xi_k)_{k\in\N}$ is independent with mean zero.  Thus, Lemma~\ref{Lem:YrCoupSource} yields a random vector $\mathfrak{g}_{\mathcal{K},n} \sim \mathcal{N}(0,\Cov[(\ta_{n}^{(1)}(v),\ta_{n}^{(2)}(v))_{v\in\mathcal{K}}])$ such that
\begin{equation}\label{eq:ApplYurAsState}
\mathbb{P}\Big( \big\|\mathfrak{g}_{\mathcal{K},n} - (\ta_n^{(1)}(v),\ta_n^{(2)}(v))_{v\in\mathcal{K}}\big\|_{\infty} >n^{-2}\Big) \le \min_{t\ge 0}\Big\{2\mathbb{P}\Big(\|Z\|_{\infty}>t \Big)+ Cn^6\beta t^2\Big\},
\end{equation}
with $Z = (Z_j)_{j\le K}$ i.i.d.\@, $\mathcal{N}(0,1)$-distributed and $\beta$ as in \cref{eq:Beta}. To control the $\xi_k$ part of $\beta$, we use that $\|\xi_k\|_2^2\le  K\|\xi_k\|_{\infty}^2$ and
$
\|\xi_k\|_{\infty} \le N^{-1/2} \max\{|a_{l_n(k)}^{(1)}|, |a_{l_n(k)}^{(2)}|\}
$
to get
\begin{equation}\label{eq:betaXiPart}
\sum_{k=1}^N \mathbb{E}[\|\xi_k\|_2^2\, \|\xi_k\|_{\infty}] \le \mathbb{E}[\max\{|a_1^{(1)}|,|a_1^{(2)}|\}]\cdot N^{-1/2}K.
\end{equation}
The $\mathfrak{g}_k$ part of $\beta$ is bounded as $CN^{-1/2}K\log(K)^{1/2}$ by \cite[Lemma~B.3]{new-yurinskii}, so that 
\[
\beta \le CN^{-1/2}K\log(K)^{1/2}.
\]
Furthermore, using Borel-TIS and $\mathbb{E}\left[ \|Z \|_{\infty}\right] \le C\log(K)^{1/2}$, we get for $t\ge C\log(K)^{1/2}$
\[
\mathbb{P}\big(\|Z\|_{\infty}>t \big) \le C e^{-\frac{(t-C\log(K)^{1/2})^2}{2}}
\]
where we recall the constant $C$ can differ from instance to instance.
Plugging this into \cref{eq:ApplYurAsState} bounds the left-hand side of \cref{eq:ApplYurAsState} by
\begin{align*}
  C\min_{t\ge C\log(K)^{1/2}}\Big\{e^{-\frac{(t-C\log(K)^{1/2})^2}{2}}+n^6t^2N^{-\frac12}K\log(K)^{\frac12}\Big\}.
\end{align*}
Taking $t = C\log(K)^{1/2}+n$ and recalling  $N^{-1/2} = 2^{-(n-1)/2}$ and $K = 2|\mathcal{K}|$ yields \cref{eq:CouplingQuality}.
\end{proof}

We now use Lemma~\ref{Lem:YurCoupl} to recursively construct the Gaussians  $(\tg_n^{(i)})_{i \in \{1,2\},n\in\N}$.
As a first step we set $(\tg_{1}^{(i)})_{i\in\{1,2\}} := \mathfrak{g}_{\cN_1,1}$ from Lemma \ref{Lem:YurCoupl}.

Assume we have constructed $(\tg_k^{(i)})_{i \in \{1,2\},k\le n}$. We construct the $(n+1)^{\mathrm{st}}$ generation as follows. First, we  only apply the coupling on points satisfying some notion of thickness ``up to level $n$'', thereby
reducing the quantity of random variables we need to couple. 
Towards this, fix $\delta>0$ satisfying $\gamma-\delta>1$, and define the following sets of \textit{time-$n$ thick points}:
\begin{align}
\mathcal{T}_{n,\gamma,\ta} &:= \left\{ v\in\mathcal{N}_n :  \tS_{n,\ta}(v)  \ge \log(2)(\gamma-\delta)n \right\}\,, \nonumber\\
\mathcal{T}_{n,\gamma,\tg} &:= \left\{v\in \mathcal{N}_n :  \tS_{n,\tg}(v) \ge \log(2)(\gamma-\delta)n \right\}\,, \nonumber \\
\mathcal{T}_{n,\gamma} &:= \mathcal{T}_{n,\gamma,\ta}\cup\mathcal{T}_{n,\gamma,\tg}\,. \label{def:time-n-thick}
\end{align}
Next, define the random set of direct descendants of particles in $\mathcal{T}_{n,\gamma}$:
\[
\mathcal{N}_{n+1}(\mathcal{T}_{n,\gamma}) := \{v\in \mathcal{N}_{n+1} : \exists w \in \mathcal{T}_{n,\gamma} \text{ such that } w \prec v \} \,.
\] 
Our coupling takes place on the random variables associated with $\mathcal{N}_{n+1}(\mathcal{T}_{n,\gamma})$, shown next.

\begin{lemma} \label{Lem:CouplThick}
There is a Gaussian random vector $(\tg_{n}^{(1)}(v),\tg_{n}^{(2)}(v))_{v\in\mathcal{N}_{n+1}(\mathcal{T}_{n,\gamma})}$ 
such that 
\[
    \Cov\big[(\tg_{n+1}^{(1)}(v),\tg_{n+1}^{(2)}(v))_{v\in\mathcal{N}_{n+1}(\mathcal{T}_{n,\gamma})}\big] = \Cov\big[(\ta_{n+1}^{(1)}(v),\ta_{n+1}^{(2)}(v))_{v\in\mathcal{N}_{n+1}(\mathcal{T}_{n,\gamma})}\big]
\]
and for all $\delta>0$
\begin{equation} \label{eq:CouplingGoodOnThick}
\mathbb{P}\Big(\sup_{v\in\mathcal{N}_{n+1}(\mathcal{T}_{n,\gamma}), i\in\{1,2\}} \Big|\tg_{n+1}^{(i)}(v)-\ta_{n+1}^{(i)}(v)\Big| > n^{-2} \Big) \le C2^{-\frac{(\gamma-\delta)^2-1}{4+\delta}}n\,.
\end{equation}
\end{lemma}
\begin{proof}
Observe that by construction $\mathcal{N}_{n+1}(\mathcal{T}_{n,\gamma})$ is independent of $(\ta_{n+1}^{(i)})_{i\in\{1,2\}}$. 
As a consequence, we may apply \cref{Lem:YurCoupl} to construct a Gaussian random vector $(\tg_{n+1}^{(1)}(v),\tg_{n+1}^{(2)}(v))_{v\in\mathcal{N}_{n+1}(\mathcal{T}_{n,\gamma})}$ with 
\[
\Cov[(\tg_{n+1}^{(1)}(v),\tg_{n+1}^{(2)}(v))_{v\in\mathcal{N}_{n+1}(\mathcal{T}_{n,\gamma})}] = \Cov[(\ta_{n+1}^{(1)}(v),\ta_{n+1}^{(2)}(v))_{k\in\mathcal{N}_{n+1}(\mathcal{T}_{n,\gamma})}]
\] 
and
\begin{multline}\label{eq:CondCouplContr}
    \mathbb{P}\Big(\sup_{v\in\mathcal{N}_{n+1}(\mathcal{T}_{n,\gamma}), i\in\{1,2\}} \big|\tg_{n+1}^{(i)}(v)-\ta_{n+1}^{(i)}(v)\big| > n^{-2} \given \mathcal{N}_{n+1}(\mathcal{T}_{n,\gamma})\Big) \\
    \leq C \big(n^8 2^{-n/2} |\mathcal{N}_{n+1}(\mathcal{T}_{n,\gamma})| \log\left(|\mathcal{N}_{n+1}(\mathcal{T}_{n,\gamma})|\right)^{3/2} \big).
\end{multline}
On the event $|\mathcal{N}_{n+1}(\mathcal{T}_{n,\gamma})| \le |\mathcal{N}_{n+1}| \, 2^{-\frac{1+(\gamma-\delta)^2}{4}n}$, the right-hand side of \cref{eq:CondCouplContr} is smaller than the right-hand side in \cref{eq:CouplingGoodOnThick}, where we made use of $|\mathcal{N}_{n+1}| \le (n+1)^{4} 2^{n+2}$. Furthermore, a first moment calculation yields
\begin{equation}\label{eq:FewThickPoints}
\mathbb{P}\Big( \big|\mathcal{N}_{n+1}(\mathcal{T}_{n,\gamma})\big| > |\mathcal{N}_{n+1}| 2^{-\frac{1+(\gamma-\delta)^2}{4}n} \Big) \leq C2^{-\frac{(\gamma-\delta)^2-1}{4}n} \,.
\end{equation}
We combine \cref{eq:CondCouplContr} and \cref{eq:FewThickPoints} to get \cref{eq:CouplingGoodOnThick}.
\end{proof}

Having constructed the Gaussian vector $(\tg_n^{(i)}(v))_{v\in \cN_n(\cT_{n-1,\gamma})}$ for each $n\in \N$ and $i\in \{1,2\}$, we now extend this vector to $(\tg_n^{(i)}(v))_{v\in \cN_n}$ in such a way that \cref{eqn:tg-ta-sameCov} is satisfied.
Using Borel-Cantelli and a first moment bound on the right tail of $|\mathcal{N}_n(\mathcal{T}_{n-1,\gamma})|$ there is a random $n_0(\omega)$ such that for all $n\ge n_0(\omega)$ we can choose $\overline{\cN_n}$ such that $|\overline{\cN_n}| =  2^{n-1}$ and
\[
  \cN_n(\cT_{n-1,\gamma}) \subset \overline{\cN_n} \subset \cN_n\,.
\]
First, we describe our coupling for $n\le n_0(\omega)-1$. In this case, we take an independent coupling, i.e., given $n_0(\omega)$ we take $(g_k^{(1)},g_k^{(2)})_{k\le 2^{n(\omega)-1}-1}$ to be an i.i.d.\@ sequence of $\mathcal{N}(0,1)$ distributed random variables and define $((\tg_n^{(i)}(v))_{v\in \cN_n}))_{n\le n_0(\omega)-1}$ analogously to \cref{Def:Coeffs}.

Next, we give our coupling for $n\ge n_0(\omega)$. In this case, we extend $(\tg_n^{(i)}(v))_{v\in \cN_n(\cT_{n-1,\gamma})}$ to a Gaussian vector $(\tg_n^{(i)}(v))_{v\in \overline{\cN_n}}$ such that 
\[
  \Cov\big[\tg_n^{(i)}(v)\big]_{v\in\overline{\cN_n}} = \Cov\big[\ta_n^{(i)}(v)\big]_{v\in\overline{\cN_n}}\,.
\]
Define
\begin{align}\label{def:iid-gaussians}
  \big(g_k^{(i)}\big)_{k \in \{2^{n-1},\dots,2^n-1\}}
  := \big(\Ci{k}{n}{v}\big)^{-1}_{k \in \{2^{n-1}, \dots 2^n-1\},v\in\overline{\cN_n}} \big(\tg_n^{(i)}(v)\big)_{v\in\overline{\cN_n}} 
  \,,
\end{align}
and note that the $g_k$ are i.i.d.\ $\cN(0,1)$ random variables (invertibility is guaranteed by our choice of the $t_n(v)$, see \eqref{def:Cn}, and the $g_k$ have identity covariance matrix due to \eqref{Def:Coeffs} and the $a_k$ being i.i.d.\ with variance $1$).
Finally, define 
\begin{align}\label{def:tilde-gaussians}
  \big(\tg_n^{(i)}(v)\big)_{v\in\cN_n} 
  = \big(\Ci{k}{n}{v}\big)_{k \in \{2^{n-1}, \dots 2^n-1\},v\in\cN_n} 
  \big(g_k^{(i)}\big)_{k \in \{2^{n-1},\dots,2^n-1\}}\,.
\end{align}
We see from \cref{Def:Coeffs} that \cref{eqn:tg-ta-sameCov} is satisfied. We also note that the $g_k^{(i)}$ constructed in \cref{def:iid-gaussians} are the ones appearing in the statement of \cref{Theo:Main}.

As mentioned at the start of the subsection, the relationship between the $g_k^{(i)}$ and the $\tg_n^{(i)}$ is the same as that of the $a_k^{(i)}$ and the $\ta_n^{(i)}$, given by \cref{Def:Coeffs}. As such, our construction fits in the framework of \cref{subsec:hierarchical-model}. By taking $a_k^{(i)} = g_k^{(i)}$ in \cref{Def:Coeffs}, we may define the field $\tS_{k,g}(v)$ as in \cref{def:tree-field}, as well as the limiting chaos measure  $\widetilde{\mu}_{\gamma, \tg}$ as in \cref{eq:PreLimHierGMC}.

\begin{proposition}\label{Prop:HierGaussAppr}
Let $\gamma\in (1,\sqrt{2})$, fix $(a_k^{(i)})_{i \in \{1,2\},k\in\N}$ i.i.d\@ satisfying \cref{eq:CondsOnak}, and let $(\ta_n^{(i)})_{i \in \{1,2\},n\in\N}$ be defined as in \eqref{Def:Coeffs}. Construct $(\tg_n^{(i)}(v))_{i \in \{1,2\},v\in \cN_n}$ as in \cref{def:tilde-gaussians}. Then,
\begin{equation}\label{eq:MutAbsContHierMod}
\widetilde{\mu}_{\gamma,\ta} \ll \widetilde{\mu}_{\gamma, \tg} \ll \widetilde{\mu}_{\gamma,\ta} \qquad \text{almost surely.}
\end{equation}
\end{proposition}

The proof of \cref{Prop:HierGaussAppr} is completed in \cref{sec:pf-HierGaussAppr}.

\subsection{Proof of \cref{Theo:Main}}
\label{subsec:pf-main-theorem}
In \cref{Prop:TrigToHier} and \cref{Prop:HierGaussAppr}, we constructed on the same probability space $(a_k^{(i)})_{i\in\{1,2\},k\in \N}$, $(\ta_n)_{n\in\N}$, $(g_k^{(i)})_{i\in\{1,2\},k\in \N}$, and $(\tg_n= \tg_n^{(1)} + \tg_n^{(2)})_{n\in\N}$ such that
\[
  \widetilde{\mu}_{\gamma, \tilde{g}}\ll \mu_{\gamma,a} \ll \widetilde{\mu}_{\gamma, \tilde{g}} \qquad \text{almost surely.}
\] 
\cref{Prop:TrigToHier} applied with $a_k^{(i)}= g_k^{(i)}$ (and thus $\ta_n = \tg_n$) yields 
\[
   \mu_{\gamma,g} \ll \widetilde{\mu}_{\gamma, \tilde{g}} \ll \mu_{\gamma,g} \qquad \text{almost surely.}
\]
The above two displays then yield \cref{Theo:Main}\,. \qed

\section{Proof of Proposition \ref{Prop:HierGaussAppr}} 
\label{sec:pf-HierGaussAppr}
We have already constructed a coupling between the non-Gaussian random variables $(\ta_n^{(i)})_{i\in\{1,2\}, n\in\N}$ with Gaussians $(\tg_n^{(i)})_{i\in\{1,2\},n\in\N}$ in \cref{subsec:Gaussian-models}. Let $(\Omega, \cF, \P)$  denote the common probability space on which these variables are constructed.
Towards  \cref{Prop:HierGaussAppr}, it remains to show the absolute continuity \cref{eq:MutAbsContHierMod}. 
The natural candidate for a Radon-Nikodym derivative between $\widetilde{\mu}_{\gamma,\ta}$ and $\widetilde{\mu}_{\gamma,\tg}$ is the $n\to\infty$ limit of 
\begin{equation}\label{def:RN-derivative}
\widetilde{R}_{n,\gamma}(t) := e^{\gamma (\tS_{n,\tg}(t)-\tS_{n,\ta}(t))} \tZ_{n,\gamma,\ta}(t)\tZ_{n,\gamma,\tg}(t)^{-1}.
\end{equation} 
Our strategy is then to show that the following hold $\mathbb{P}$-almost surely:
\begin{enumerate}[(A)]
\item \label{it:RNconv}$\widetilde{R}_{n,\gamma}(t)$ converges  to a value in $(0,\infty)$ for $\widetilde{\mu}_{\gamma,\ta}$-almost every and $\widetilde{\mu}_{\gamma,\tg}$-almost every $t \in [0,1]$,
\item \label{it:SetConv} for all $A\in\mathcal{B}([0,1])$ we have
\begin{align}
    \lim\limits_{n\to\infty} \int_A \widetilde{R}_{n,\gamma}(t)\, \mathrm{d}\widetilde{\mu}_{\gamma, \ta}(t) &= \widetilde{\mu}_{\gamma,\tg}[A] \,, \label{eq:IntegralLimit}\\
    \lim\limits_{n\to\infty} \int_A \widetilde{R}_{n,\gamma}(t)^{-1}\, \mathrm{d}\widetilde{\mu}_{\gamma, \tg}(t) &= \widetilde{\mu}_{\gamma,\ta}[A] \,, \text{ and} \label{eq:IntegralLimit2}
\end{align}
\item \label{it:RNUnifInt}$(\widetilde{R}_{n,\gamma}(t))_{n\in\N}$ is uniformly integrable with respect to $\widetilde{\mu}_{\gamma,\ta}$ and $(\widetilde{R}_{n,\gamma}(t)^{-1})_{n\in\N}$ is uniformly integrable with respect to $\widetilde{\mu}_{\gamma,\tg}$.
\end{enumerate}
Before we establish  \eqref{it:RNconv}--\eqref{it:RNUnifInt}, we quickly show how they imply Proposition~\ref{Prop:HierGaussAppr}.

\begin{proof}[Proof of Proposition~\ref{Prop:HierGaussAppr} given \eqref{it:RNconv}--\eqref{it:RNUnifInt}]
We show  $\widetilde{\mu}_{\gamma,\tg}\ll \widetilde{\mu}_{\gamma,\ta}$; the other direction follows similarly.

Define $\widetilde{R}_{\infty,\gamma}(t) := \lim\limits_{n\to\infty} \widetilde{R}_{n,\gamma}(t)$, and note that by \eqref{it:RNconv} this limit exists for $\widetilde{\mu}_{\gamma,\ta}$-almost every $t\in[0,1]$. This along with \eqref{it:RNUnifInt} allows us to apply Vitali's convergence theorem, which yields $\widetilde{R}_{\infty,\gamma}\in L^1(\widetilde{\mu}_{\gamma,\ta})$ and, for each $A\in\mathcal{B}([0,1])$,
\[
\lim\limits_{n\to\infty} \int_A\widetilde{R}_{n,\gamma}(t)\;\mathrm{d}\widetilde{\mu}_{\gamma,\ta}(t) = \int_A \widetilde{R}_{\infty,\gamma}(t)\;\mathrm{d}\widetilde{\mu}_{\gamma,\ta}(t) \,.
\]
Combining the last display with \eqref{it:SetConv} yields 
\[
\widetilde{\mu}_{\gamma,\tg}[A] = \int_A \widetilde{R}_{\infty,\gamma}(t)\;\mathrm{d}\widetilde{\mu}_{\gamma,\ta}(t) \,,
\]
thereby establishing the desired absolute continuity.
\end{proof}

We prove items \eqref{it:RNconv}--\eqref{it:RNUnifInt} in the next three subsections.

\subsection{Proof of \eqref{it:RNconv}: convergence of the Radon-Nikodym derivative}
\Cref{eqn:junnila-lemma-7} implies $\tZ_{n,\gamma, \ta}(t)\tZ_{n,\gamma,\tg}(t)^{-1}$ converges in $(0,\infty)$ as $n\to\infty$ for every $t\in [0,1]$.
Thus, we need only show the convergence of $e^{ \gamma (\tS_{n,\tg}(t)-\tS_{n,\ta}(t))}$.
This will come from the coupling estimate along thick points (\cref{Lem:CouplThick}), as well as \cref{Lem:SupportThick} below, which states that the chaos measures are carried by thick points.
Towards this, set
\[
\mathcal{T} := \bigg\{t\in [0,1] \,:\, \lim\limits_{n\to\infty} \frac{\tS_{n,\ta}(t)}{n} = \log(2)\gamma\bigg\} \cup \bigg\{t\in [0,1] \,:\, \lim\limits_{n\to\infty} \frac{\tS_{n,\tg}(t)}{n} = \log(2)\gamma\bigg\}\,.
\]
\begin{lemma}\label{Lem:SupportThick}
    We have $\mathbb{P}$-almost surely $\widetilde{\mu}_{\gamma,\ta}[\mathcal{T}^c]+\widetilde{\mu}_{\gamma,\tg}[\mathcal{T}^c] = 0$. 
\end{lemma}
\cref{Lem:SupportThick} is proved in Appendix~\ref{appendix:thick}. 
We say $t\in\mathcal{T}_{n,\gamma}$ if and only if $v_n(t)\in \mathcal{T}_{n,\gamma}$. Recall also the notation set forth in the display above \eqref{def:tree-field}. By definition, we have
\begin{align*}
\mathcal{T}\subseteq \bigcup_{N\in\N} \bigcap_{n\ge N} \mathcal{T}_{n,\gamma}\,.
\end{align*}
Thus, Lemma~\ref{Lem:SupportThick} implies it suffices to show  $\widetilde{R}_{n,\gamma}(t)$ converges for $t\in\bigcup_{N\in\N} \bigcap_{n\ge N} \mathcal{T}_{n,\gamma}$. Let $t\in \bigcup_{N\in\N}\bigcap_{n\ge N} \mathcal{T}_{n,\gamma}$ be arbitrary.  There exists exists a $N_t(\omega)\in\N$ such that $t\in \bigcap_{n\ge N_t(\omega)} \mathcal{T}_{n,\gamma}$. Furthermore,  \cref{eq:CouplingGoodOnThick} and Borel-Cantelli imply that for each $\omega \in \Omega$, there is an $N_0(\omega)$ such that for all $n\ge N_0(\omega)$ and $t\in \mathcal{T}_{n,\gamma}$ we have $|\tg_{n}(t)-\ta_n(t)| \le n^{-2}$. Set $\widetilde{N}_t(\omega) := \max\{N_0(\omega), N_t(\omega)\}$. For $n\ge \widetilde{N}_t(\omega)$, we have
\[
\sum_{m=1}^n |\tg_m(t)-\ta_m(t)|\le \sum_{m=1}^{\widetilde{N}_t(\omega)} \left|\tg_{m}(t)-\ta_{m}(t)\right|+\sum_{m = \widetilde{N}_t(\omega)}^{n} m^{-2}\le C_t(\omega) <\infty\,.
\]
In particular, $\sum_{m=1}^{n} (\tg_{m}(t)-\ta_{m}(t))$ converges absolutely, which implies  $\widetilde{R}_{n,\gamma}(t)$ converges to a value in $(0,\infty)$. Since $t\in \bigcup_{N\in\N}\bigcap_{n\ge N} \mathcal{T}_{n,\gamma}$ was arbitary, this implies that $\widetilde{R}_{n,\gamma}(t)$ converges $\widetilde{\mu}_{\gamma,\ta}$-almost everywhere and $\widetilde{\mu}_{\gamma,\tg}$-almost everywhere to a value in $(0,\infty)$. \qed

\subsection{Proof of \eqref{it:SetConv}: integral limits} \label{Sec:IntLims}
We only show the first equation \cref{eq:IntegralLimit}, as the proof of \cref{eq:IntegralLimit2} can be done analogously.

\subsubsection{Reduction, and overview of the proof}
We begin by reducing \cref{eq:IntegralLimit} to more tractable limit statements, providing an overview of the proof strategy along the way.

Let us define
\begin{multline*}
  \tZ_{n,\gamma,\tg}(w) := \tZ_{n,\gamma,\tg}(t_n(w))  \text{ for $w \in \cN_n$, }\\
   \text{and} \qquad I_m(v) := I_m(j)  \text{ for $v \in \cN_m$ verifying $t_m(v) \in I_m(j)$.}
\end{multline*}
It suffices to show \cref{eq:IntegralLimit} for sets of the form $A := I_{m}(v)$, for all $m\in\N$ and $v\in \mathcal{N}_m$.
Thus, we fix $m\in\N$ and $v\in \mathcal{N}_m$, and set $\mathcal{N}_n(v)$ to be the set of particles in generation $n$ who are descendants of $v$.  Using that $\widetilde{R}_{n,\gamma}(t)$ is constant on each of the $I_{n}(w)$, $w\in\mathcal{N}_n(v)$, we get
\begin{align*}
\widetilde{\mu}_{\gamma,\tg}[I_m(v)] &=  \lim\limits_{n\to\infty}\sum_{w\in\mathcal{N}_n(v)} |I_n(w)| \, e^{\gamma \tS_{n,\tg}(w)}\tZ_{n,\gamma,\tg}(w)^{-1} \,, \text{ and}\\
\lim\limits_{n\to\infty} \int_{I_m(v)} \widetilde{R}_{n,\gamma}(t) \mathrm{d}\widetilde{\mu}_{\gamma,\ta}(t) &= \lim\limits_{n\to\infty} \sum_{w\in\mathcal{N}_n(v)} \Big(\widetilde{\mu}_{\gamma,\ta}[I_n(w)]e^{-\gamma \tS_{n,\ta}(w)}\tZ_{n,\gamma,\ta}(w) \Big) \\
&\quad \qquad\times e^{\gamma \tS_{n,\tg}(w)}\tZ_{n,\gamma,\tg}(w)^{-1} \,.
\end{align*}
For $w\in\mathcal{N}_n$, define the quantity
\begin{align*}
    M_{n,\ta}(w) &:= |I_n(w)|^{-1}\widetilde{\mu}_{\gamma,\ta}[I_n(w)]e^{-\gamma \tS_{n,\ta}(w)}\tZ_{n,\gamma,\ta}(w)  \\
    &= \lim\limits_{r\to\infty} |I_n(w)|^{-1}\int_{I_n(w)} e^{\gamma \tS_{(n,n+r],\ta}(t)}\tZ_{(n,n+r],\gamma,\ta}(t)^{-1} \d t\,,
\end{align*}
where 
\[
    \tS_{(n,n+r],\ta}(t) := \tS_{n+r,\ta}(t)-\tS_{n,\ta}(t)\quad\text{and}\quad \tZ_{(n,n+r],\gamma,\ta}(t) := \mathbb{E}[e^{\gamma \tS_{(n,n+r],\ta}(t)}]\,.
\] 
It is helpful to observe that $M_{n,\ta}(w)$ is the total weight of the hierarchical chaos measure $\widetilde{\mu}_{\gamma,\ta}$ on the subtree of $\mathfrak{T}$ rooted at $w$, rescaled (by $|I_n(w)|^{-1}$) so that $\mathbb{E}[M_{n,\ta}(w)] = 1$. 
Further, as a consequence of the ``generational independence'' (that is,  independence of the $\ta_n^{(i)}$ across different $n$), the random vector $(M_{n,\ta}(w))_{w\in\mathcal{N}_n(v)}$ is independent of $ (e^{\gamma \tS_{n,\tg}(w)}\tZ_{n,\gamma,\tg}(w)^{-1})_{w\in\mathcal{N}_n(v)}$. This fact is used numerous times below.

Using the definition of $M_{n,\ta}(w)$, it suffices to show  the following holds $\mathbb{P}$-almost surely:
\begin{equation}\label{eq:GoalIntegralConvergence}
\lim\limits_{n\to\infty} \sum_{w\in\mathcal{N}_n(v)} |I_n(w)|\, (M_{n,\ta}(w)-1) \, e^{\gamma \tS_{n,\tg}(w)}\tZ_{n,\gamma,\tg}(w)^{-1} = 0\,.
\end{equation}

\begin{remark}[Proof idea]
Since $|I_n(w)| \approx |\mathcal{N}_n(v)|^{-1}$
and  due to the aforementioned independence $\mathbb{E}[  (M_{n,\ta}(w)-1) e^{\gamma \tS_{n,\tg}(w)}\tZ_{n,\gamma,\tg}(w)^{-1} ] = 0$, the above roughly corresponds to a law of large numbers. 
However, all terms in the above sum are correlated to various degrees.
We prove \cref{eq:GoalIntegralConvergence} by moment bounds on the sum that give the desired convergence upon applying Markov's inequality and Borel-Cantelli.
An issue with this approach is that the first moment does not decay, while $\mathbb{E}[M_{n,\ta}(w)^2] = \infty$ and $\mathbb{E}[(e^{\gamma \tS_{n,\tg}(w)}\tZ_{n,\gamma,\tg}(w)^{-1})^2]$ grows too fast with $n$ (in light of \cref{eqn:junnila-lemma-7}).  To deal with this, we truncate both objects by instituting   ``barrier events''. We  need a barrier  both for $\tS_{s,\tg}(w)$, $s\leq n$, and for $\tS_{(n,n+r],\ta}(t)$, which is the field below generation $n$. The differences between the truncated and un-truncated quantities can be shown to be small in $L^1$ by mimicking the computations done in \cite{BerestyckiSimplePath}.
\end{remark}

The first step is truncating $M_{n,\ta}(w)$ by instituting a barrier on $\tS_{(n,n+r],\ta}(t)$. 
Fix $\alpha\in (\gamma,\sqrt{2})$, and define the ``barrier event''
\[
B_{\log(n)^2,r;n}^\alpha(t) := \Big\{ \forall\, l\in \{\log(n)^2,\dots,  r\} \,:\, \tS_{(n,n+l],\ta}(t) \le \alpha \log(2) l\Big\}\,, \text{ for } t\in I_n(w)\,.
\]
This is the event that the field after generation $n$ is not too thick, starting from generation $n+\log(n)^2$. We then define
\begin{align}
  M_{n,\ta}^{(1)}(w) &:= \lim\limits_{r\to\infty} |I_n(w)|^{-1}\int_{I_n(w)} \indset{B_{\log(n)^2,r;n}^\alpha(t)} e^{\gamma\tS_{(n,n+r],\ta}(t)}\tZ_{(n,n+r],\gamma,\ta}(t)^{-1}\;\mathrm{d}t, \label{eq:MnMainPart}\\
  M_{n,\ta}^{(2)}(w) &:= \lim\limits_{r\to\infty} |I_n(w)|^{-1}\int_{I_n(w)} \indset{(B_{\log(n)^2,r;n}^\alpha(t))^c}e^{\gamma \tS_{(n,n+r],\ta}(t)}\tZ_{(n,n+r],\gamma,\ta}(t)^{-1}\;\mathrm{d}t \,, \label{eq:MnBadL2Part}
\end{align}
where the $\mathbb{P}$-almost sure limit in \cref{eq:MnMainPart} exists, since 
\begin{equation}\label{eq:PosSuperMart}
\Big(|I_n(w)|^{-1}\int_{I_n(w)} \indset{B_{\log(n)^2,r;n}^\alpha(t)} e^{\gamma\tS_{(n,n+r],\ta}(t)}\tZ_{(n,n+r],\gamma,\ta}(t)^{-1}\;\mathrm{d}t\Big)_{r\in\N}
\end{equation}
is a positive supermartingale. The $\mathbb{P}$-a.s.\@ limit in \cref{eq:MnBadL2Part} exists, since it is the limit of the  difference between the positive martingale $(|I_n(w)|^{-1}\int_{I_n(w)} e^{\gamma\tS_{(n,n+r],\ta}(t)}/\tZ_{(n,n+r],\gamma,\ta}(t)\;\mathrm{d}t)_{r}$ and the positive supermartingale in \cref{eq:PosSuperMart}. We note that, by definition, $M_{n,\ta}^{(1)}(w)+M_{n,\ta}^{(2)}(w) = M_{n,\ta}(w)$. For $i\in\{1,2\}$, we define the centered variants of these:
\[
\bar{M}_{n,\ta}^{(i)}(w) := M_{n,\ta}^{(i)}(w)-\mathbb{E}[M_{n,\ta}^{(i)}] \,.
\] 
While the contribution of $M_{n,\ta}^{(2)}(w)$ can be controlled with a first moment argument, we also need a truncation of $\tS_{s,\tg}(w)$, $s\leq n$, to control the contribution of $M_{n,\ta}^{(1)}(w)$. For this, we set for $w\in\mathcal{N}_n(v)$
\begin{align*}
B_{h,n}^{\alpha}(w) &:= \big\{\forall s\in\{h+1,\dots, n\}\,:\, \tS_{s,\tg}(w) \le \alpha\log(2) s\big\} \,, \text{ and}\\
A_h^{\alpha}(w) &:= \{\tS_{h,\tg}(w)>\alpha\log(2)h\} \,.
\end{align*}
Finally, recalling that $\E[M_{n,\ta}^{(1)}(w)] + \E[M_{n,\ta}^{(2)}(w)] = \E[M_{n,\ta}(w)] = 1$ for all $w\in\mathcal{N}_n(v)$, \cref{eq:GoalIntegralConvergence} (and thus \cref{eq:IntegralLimit}) follows from the existence of some $\alpha \in (\gamma, \sqrt2)$ such that $\mathbb{P}$-a.s.:
\begin{align}
    \lim\limits_{n\to\infty} \sum_{w\in\mathcal{N}_n(v)} |I_n(w)| \, \bar{M}_{n,\ta}^{(1)}(w)e^{\gamma \tS_{n,\tg}(w)}\tZ_{n,\gamma,\tg}(w)^{-1}\indset{B_{\log(n)^{1.5},n}^{\alpha}(w)}&=0 \,, \label{eq:L2PartConv}\\
    \lim\limits_{n\to\infty} \sum_{w\in\mathcal{N}_n(v)} |I_n(w)| \, \bar{M}_{n,\ta}^{(1)}(w)e^{\gamma \tS_{n,\tg}(w)}\tZ_{n,\gamma,\tg}(w)^{-1}\indset{B_{\log(n)^{1.5},n}^{\alpha}(w)^c}&=0\,, \text{ and} \label{eq:L2PartConvSec}\\
    \lim\limits_{n\to\infty} \sum_{w\in\mathcal{N}_n(v)} |I_n(w)| \, \bar{M}_{n,\ta}^{(2)}(w)e^{\gamma \tS_{n,\tg}(w)}\tZ_{n,\gamma,\tg}(w)^{-1} &= 0\,. \label{eq:L1PartConv}
\end{align}
\Cref{eq:L2PartConv} considers a barrier on the increments of the field $\tS_{\cdot,\tg}(w)$ before level $n$ from the  $B^{\alpha}_{\log(n)^{1.5},n}(w)$ term, as well as a barrier on the increments of $\widetilde{S}_{(n,\,\cdot\,],\widetilde{a}}(w)$ after level $n$ from the $\bar{M}_{n,\ta}^{(1)}(w)$ term.  It will be shown by controlling the second moment of the sum, and its proof will occupy the bulk of the remainder of the article.
\Cref{eq:L2PartConvSec,eq:L1PartConv} consider the increment of the field $\widetilde{S}_{\cdot,\widetilde{g}}(w)$ crossing a barrier before level $n$ respectively the increment of the field $\widetilde{S}_{(n,\,\cdot\,],\widetilde{a}}(w)$ crossing a barrier after level $n$; each will be shown by controlling the first moment.

\subsubsection{Proofs of \cref{eq:L2PartConv,eq:L2PartConvSec,eq:L1PartConv}}
As mentioned above, we aim to bound first and second moments. This requires the following moment bounds on the $M_{n,\ta}^{(i)}(w)$.
\begin{lemma}\label{Lem:SecondMoment}
There is an $\alpha_0\in (\gamma,\sqrt{2})$ such that for all $\alpha \in (\gamma,\alpha_0)$ and all $w\in\mathcal{N}_n(v)$,
\[
\mathbb{E}\, \big|\bar{M}_{n,\ta}^{(1)}(w)\big|^2 \le C\log(n)^2 2^{\gamma^2\log(n)^2}\,.
\]
\end{lemma}

\begin{lemma}\label{Lem:FirstMoment}
For any $\alpha >\gamma$ and $w\in\mathcal{N}_n(v)$, 
\[
\mathbb{E}\Big[M_{n,\ta}^{(2)}(w)\Big] \le C2^{-\frac{(\alpha-\gamma)^2}{2}\log(n)^2}\,.
\]
\end{lemma}

\begin{lemma}\label{Lem:CovarTerm}
For any $\alpha > \gamma$  and $w_1,w_2 \in\mathcal{N}_n(v)$ satisfying $|t_n(w_1)-t_n(w_2)| \in [2^{-n}f(n),1/8]$, 
\begin{align*}
\mathbb{E}\Big[ \bar{M}_{n,\ta}^{(1)}(w_1)\bar{M}_{n,\ta}^{(1)}(w_2)\Big] \leq Cn^{-2}\,.
\end{align*}
Furthermore, if $|t_n(w_1)-t_n(w_2)| \in [2^{-n+\log(n)^2}, 1/8]$, 
\begin{equation}\label{eq:SecondCov}
\mathbb{E}\Big[ \bar{M}_{n,\ta}^{(1)}(w_1)\bar{M}_{n,\ta}^{(1)}(w_2)\Big] \le C2^{-\min(\frac{(\alpha-\gamma)^2}{2},1)\log(n)^2}\,.
\end{equation}
\end{lemma}

Before proving the lemmata, we show how they imply \cref{eq:L2PartConv,eq:L2PartConvSec,eq:L1PartConv} (in reverse order). We use the independence of $(M^{(i)}_{n,\ta}(w))_{i\in\{1,2\},w\in\mathcal{N}_n(v)}$ from $ (e^{\gamma \tS_{n,\tg}(w)}\tZ_{n,\gamma,\tg}(w)^{-1})_{w\in\mathcal{N}_n(v)}$ repeatedly, referring to this fact as ``generational independence''.

\begin{proof}[Proof of \cref{eq:L1PartConv}]
Using $\E[e^{\gamma \tS_{n,\tg}(w)}\tZ_{n,\gamma,\tg}(w)^{-1}] =1$, generational independence, as well as Lemma \ref{Lem:FirstMoment}, we obtain the first moment bound
\begin{multline*}
\mathbb{E}\bigg|\sum_{w\in\mathcal{N}_n(v)} |I_n(w)| \, \bar{M}_{n,\ta}^{(2)}(w)e^{\gamma \tS_{n,\tg}(w)}\tZ_{n,\gamma,\tg}(w)^{-1}\bigg| \\
\le \sum_{w\in\mathcal{N}_n(v)} |I_n(w)| 2^{-\frac{(\alpha-\gamma)^2}{2}\log(n)^2} = |I_m(v)| 2^{-\frac{(\alpha-\gamma)^2}{2}\log(n)^2} \,.
\end{multline*}
Since $m$ is fixed, the Markov inequality and Borel-Cantelli yield  \cref{eq:L1PartConv}.
\end{proof}

\begin{proof}[Proof of \cref{eq:L2PartConvSec}] 
We again proceed by bounding the first moment. The decay will follow from the $\tS_{n,\tg}(w)$ term and the complement of its barrier event, so we crudely bound $\E |\bar{M}_{n,\ta}^{(1)}(w)| \leq 2\E[M_{n,\ta}(w)] = 2$ to obtain:
\begin{multline*}
  \mathbb{E}\bigg|\sum_{w\in\mathcal{N}_n(v)} |I_n(w)| \bar{M}_{n,\ta}^{(1)}(w)e^{\gamma \tS_{n,\tg}(w)}\tZ_{n,\gamma,\tg}(w)^{-1}\indset{B_{\log(n)^{1.5},n}^{\alpha}(w)^c}\bigg| \\
  \le 2|I_m(v)| \max_{w\in\mathcal{N}_n(v)}\mathbb{E}\Big[e^{\gamma \tS_{n,\tg}(w)}\tZ_{n,\gamma,\tg}(w)^{-1}\indset{B_{\log(n)^{1.5},n}^{\alpha}(w)^c}\Big]\,.
\end{multline*}
Next, we decompose the complement of the barrier event based on what generation $\tS_{\cdot, \tg}(w)$ crosses the barrier. A union bound on the right-hand side above yields
\begin{multline}\label{eq:L2PartConvSecStep1}
 2|I_m(v)| \max_{w\in\mathcal{N}_n(v)} \sum_{\ell=\log(n)^{1.5}+1}^n \mathbb{E}\Big[\indset{A_\ell^\alpha(w)} e^{\gamma \tS_{n,\tg}(w)}\tZ_{n,\gamma,\tg}(w)^{-1}\Big] \\
  = 2|I_m(v)|\max_{w\in\mathcal{N}_n(v)} \sum_{\ell=\log(n)^{1.5}+1}^n \mathbb{E}\Big[\indset{A_\ell^\alpha(w)}e^{\gamma \tS_{\ell,\tg}(w)}\tZ_{\ell,\gamma,\tg}(w)^{-1}\Big] \,,
\end{multline}
where the equality comes from the independence of $\tS_{n,\tg}(w)-\tS_{\ell,\tg}(w)$ from $\tS_{\ell,\tg}(w)$, which also gives
\[
  \mathbb{E}\Big[e^{\gamma (\tS_{n,\tg}(w)-\tS_{\ell,\tg}(w))}\Big] = \tZ_{n,\gamma,\tg}(w)\tZ_{\ell,\gamma,\tg}(w)^{-1} \,.
\]
By definition, 
\begin{equation}\label{eq:IndBadBoundExp}
\indset{A_\ell^\alpha(w)} \le e^{(\alpha-\gamma) \tS_{\ell,\tg}(w)-(\alpha-\gamma)\alpha \log(2)\ell}.
\end{equation}
Plugging \cref{lem:junnila-lemma-7} and \cref{eq:IndBadBoundExp} into \cref{eq:L2PartConvSecStep1}
yields
\begin{multline*}
  \mathbb{E}\bigg|\sum_{w\in\mathcal{N}_n(v)} |I_n(w)| \, \bar{M}_{n,\ta}^{(1)}(w)e^{\gamma \tS_{n,\tg}(w)}\tZ_{n,\gamma,\tg}(w)^{-1}\indset{B_{\log(n)^{1.5},n}^{\alpha}(w)^c}\bigg|\\
  \begin{aligned}
  &\le C\max_{w\in\mathcal{N}_n(v)} \sum_{\ell=\log(n)^{1.5}+1}^n \mathbb{E}\big[ e^{\alpha\tS_{\ell,\tg}(w)}\big]e^{-(\alpha-\gamma)\alpha \log(2) \ell-\frac{\gamma^2}{2}\log(2) \ell}\\
  &\le C\sum_{\ell=\log(n)^{1.5}+1}^n 2^{-\frac{(\alpha-\gamma)^2}{2}\ell} \leq Ce^{-\frac{(\alpha-\gamma)^2}{2}\log(n)^{1.5}}\,.
  \end{aligned}
\end{multline*}
\Cref{eq:L2PartConvSec} now follows from the Markov inequality and Borel-Cantelli.
\end{proof}

\begin{proof}[Proof of \cref{eq:L2PartConv}] To shorten the displays herein, we introduce the notation
\[
\widetilde{E}_{n,\gamma,\tg}(w) := e^{\gamma \tS_{n,\tg}(w)}\tZ_{n,\gamma,\tg}(w)^{-1}\,,\qquad \widetilde{E}_{(k,n],\gamma,\tg}(w) := e^{\gamma\tS_{(k,n],\tg}(w)}\tZ_{(k,n],\gamma,\tg}(w)^{-1},
\]
for $k,n\in\N$, $k\le n$, $w\in\mathcal{N}_n$.  We analogously define $\widetilde{E}_{n,\gamma,\ta}(w)$, $\widetilde{E}_{(k,n],\gamma,\ta}(w)$.

We proceed by bounding the second moment of the sum in \cref{eq:L2PartConv}. Our strategy will be to repartition the sum over $\mathcal{N}_n(v)$ such that the $\bar{M}_{n,\ta}^{(1)}(w)$ in each partition have small correlation. Towards this, keeping in mind Lemma~\ref{Lem:CovarTerm}, we write the disjoint union 
\[
    \mathcal{N}_n(v) = \bigsqcup_{j=1}^{Cf(n)^2} \mathcal{N}_{n,j}(v)
\]
in such a way that  any $w_1\neq w_2$ in $\mathcal{N}_{n,j}(v)$ satisfies $1/8\ge |t_n(w_1)-t_n(w_2)| \ge 2^{-n}f(n)$. Further decomposing according to the last level $h\le \log(n)^{1.5}$ at which $\tS_{h,\tg}(w)> \alpha \log(2)h$, 
we then express the pre-limit in the left-hand side of \cref{eq:L2PartConv} as
\begin{align*}
\sum_{j=1}^{Cf(n)^2} \sum_{h=0}^{\log(n)^{1.5}} \sum_{w\in\mathcal{N}_{n,j}(v)}  |I_{n}(w)| \, \bar{M}_{n,\ta}^{(1)}(w)\widetilde{E}_{n,\gamma,\tg}(w)\indset{A_h^\alpha(w)}\indset{B_{h,n}^{\alpha}(w)} \,.
\end{align*}
We now compute the second moment of the above. Below, we use the fact $(\sum_{i=1}^N |x_i|)^2 \le N\sum_{i=1}^N |x_i|^2$ and the triangle inequality to pull the sums over $h$ and $j$ out of the square:
\begin{align}\label{eq:ToControllSecondMoment}
    &\mathbb{E} \, \Bigg|\sum_{j=1}^{Cf(n)^2} \sum_{h=0}^{\log(n)^{1.5}}\sum_{w\in\mathcal{N}_{n,j}(v)}  |I_{n}(w)| \, \bar{M}_{n,\ta}^{(1)}(w)\widetilde{E}_{n,\gamma,\tg}(w)\indset{A_h^\alpha(w)}\indset{B_{h,n}^{\alpha}(w)}\Bigg|^2 \nonumber\\
    &\qquad\le Cf(n)^2\log(n)^{1.5} \nonumber \\
    &\quad\qquad \times\sum_{j=1}^{Cf(n)^2}\sum_{h=0}^{\log(n)^{1.5}} \mathbb{E}\, \Bigg|\sum_{w\in\mathcal{N}_{n,j}(v)}  |I_{n}(w)| \, \bar{M}_{n,\ta}^{(1)}(w)\widetilde{E}_{n,\gamma,\tg}(w)\indset{A_h^\alpha(w)}\indset{B_{h,n}^{\alpha}(w)}\Bigg|^2 \nonumber\\
    &\qquad\le Cf(n)^4\log(n)^{3}\bigg( \max_{j\le Cf(n)^2,h \le \log(n)^{1.5}}
    \mathsf{Diag}(h,j)
    + \max_{j\le Cf(n)^2, h\le \log(n)^{1.5}}\mathsf{Cross}(h,j) \bigg)
    \end{align}
where
\[
    \mathsf{Diag}(h,j) := 
    \sum_{w\in\mathcal{N}_{n,j}(v)} |I_{n}(w)|^2 \, \mathbb{E}\Big[\bar{M}_{n,\ta}^{(1)}(w)^2\Big] \, \mathbb{E}\Big[\widetilde{E}_{n,\gamma,\tg}(w)^2\indset{A_h^\alpha(w)}\indset{B_{h,n}^{\alpha}(w)} \Big]
\]
and
\begin{multline*}
    \mathsf{Cross}(h,j):= \sum_{w_1\neq w_2 \in\mathcal{N}_{n,j}(v)} |I_{n}(w_1)| \, |I_{n}(w_2)| \, \mathbb{E}\Big[\bar{M}_{n,\ta}^{(1)}(w_1)\bar{M}_{n,\ta}^{(1)}(w_2) \Big] \\
    \times \mathbb{E}\Big[\widetilde{E}_{n,\gamma,\tg}(w_1) \widetilde{E}_{n,\gamma,\tg}(w_2)\indset{A_h^\alpha(w_1)\cap B_{h,n}^{\alpha}(w_1)}\indset{A_h^\alpha(w_2)\cap B_{h,n}^{\alpha}(w_2)} \Big] \,. 
\end{multline*}
We begin by bounding $\mathsf{Diag}(h,j)$. Let $\mathcal{F}_{\ell}$ be the filtration generated by all random variables up to the $l^{\mathrm{th}}$ level in the tree.
Conditioning on $\cF_h$, then applying \cite[Lemma 10]{Junnila} and \cref{lem:junnila-lemma-7} yields that, for all $j\le Cf(n)^2$, $h\le\log(n)^{1.5}$, and $w\in \mathcal{N}_{n,j}(w)$, we have 
\begin{align*}
    \mathbb{E}\Big[\widetilde{E}_{n,\gamma,\tg}(w)^2 \indset{A_h^\alpha(w)\cap B_{h,n}^{\alpha}(w)} \Big]
    &= \E\Big[\widetilde{E}_{h,\gamma,\tg}(w)^2 \indset{A_h^\alpha(w)} \E\Big[\widetilde{E}_{(h,n],\gamma,\tg}(w)^2 \indset{B_{h,n}^{\alpha}(w)} \given \cF_h\Big]\Big]
    \\
    &\le C \mathbb{E}\Big[\widetilde{E}_{h,\gamma,\tg}(w)^2\Big] 2^{(\gamma^2-\frac{(2\gamma-\alpha)^2}{2})(n-h)}\\
    &\le C 2^{\frac{\gamma^2}{2}n+\varepsilon(\alpha)n + \gamma^2\log(n)^{1.5}} \,,
\end{align*}
where $\ep(\alpha)\searrow 0$ as $\alpha\searrow \gamma$, and we have taken $\alpha$ close to $\gamma$. This with  \cref{Lem:SecondMoment} yields
\begin{align}\label{eq:DiagTerm}
    \max_{j\le Cf(n)^2,h \le \log(n)^{1.5}}
    \mathsf{Diag}(h,j) 
    &\leq C|\cN_{n,j}(v)| 2^{-2n} f(n)^{-2}  \log (n)^2  2^{\frac{\gamma^2}{2}n+\varepsilon(\alpha)n +2\gamma^2\log(n)^{2}} \nonumber \\
    &\leq 
    C 2^{(\frac{\gamma^2}{2}+\varepsilon(\alpha) -1)n + 2\gamma^2 \log(n)^2}\,,
\end{align}
where in the second inequality we used $|\cN_{n,j}(v)| \leq T_n$, followed by \cref{eqn:size-Tn-Inj}.

Next, we control $\mathsf{Cross}(h,j)$. For this, we further partition the sum with respect to the distance between $w_1$ and $w_2$, which warrants the definition
\begin{align*}
\mathcal{D}_{\ell,j} &:= \big\{(w_1,w_2) \in \mathcal{N}_{n,j}(v) : |t(w_1)-t(w_2)| \in [2^{-\ell-1},2^{-\ell}] \big\} \,.
\end{align*}
Set $\ell_0 := \ell_0(n) = \min\{ \ell \in\N : 2^{-\ell} < 2^{-n}f(n)\}$. By definition of $\mathcal{N}_{n,j}(v)$ we have 
\[
\{w_1 \neq w_2 \in \cN_{n,j}(v) \}= \bigsqcup_{\ell=3}^{\ell_0} \mathcal{D}_{l,j}.
\]
To shorten expressions, introduce also 
\[
    \widetilde{\cB}_{s,\gamma,\tg}^{r,\alpha}(w) := \widetilde{E}_{s,\gamma,\tg}(w) \indset{A_r^\alpha(w)\cap B_{r,s}^{\alpha}(w)}\,.
\]
Then, fixing $j\le Cf(n)^2$ and $h\le \log(n)^{1.5}$, we have
\begin{multline}\label{eq:ScaleDecomp}
 \mathsf{Cross}(h,j)=\\
    \sum_{\ell=3}^{\ell_0} \smashoperator[r]{\sum_{(w_1,w_2)\in \mathcal{D}_{\ell,j}}} |I_{n}(w_1)| \, |I_{n}(w_2)|\mathbb{E}\Big[\bar{M}_{n,\ta}^{(1)}(w_1)\bar{M}_{n,\ta}^{(1)}(w_2) \Big]
    \mathbb{E}\Big[\widetilde{\cB}_{n,\gamma,\tg}^{h,\alpha}(w_1)\widetilde{\cB}_{n,\gamma,\tg}^{h,\alpha}(w_2)\Big].
\end{multline}
 
We first consider summands with $\ell\ge h$ in \cref{eq:ScaleDecomp} (note for $\ell>\log(n)^{1.5}$, this is automatic). For any $(w_1, w_2)\in \mathcal{D}_{\ell,j}$, we find
\begin{align}\label{eq:DistanceSmallerGood1}
    \mathbb{E}\Big[&\widetilde{\cB}_{n,\gamma,\tg}^{h,\alpha}(w_1)
    \widetilde{\cB}_{n,\gamma,\tg}^{h,\alpha}(w_2) \Big]\nonumber \\
     &= \mathbb{E}\Big[\widetilde{\cB}_{\ell,\gamma,\tg}^{h,\alpha}(w_1)\widetilde{\cB}_{\ell,\gamma,\tg}^{h,\alpha}(w_2)
    \, \mathbb{E}\Big[\widetilde{E}_{(\ell,n],\gamma,\tg}(w_1)\widetilde{E}_{(\ell,n],\gamma,\tg}(w_2)\indset{B_{\ell+1,n}^{\alpha}(w_1)\cap B_{l+1,n}^{\alpha}(w_2)}  | \mathcal{F}_\ell\Big] \Big] \nonumber\\
    &\le  C\mathbb{E}\Big[\widetilde{\cB}_{\ell,\gamma,\tg}^{h,\alpha}(w_1)\widetilde{\cB}_{\ell,\gamma,\tg}^{h,\alpha}(w_2)\Big]\,,
\end{align}
where the last step uses \cite[Lemma~9]{Junnila} and  $|t_n(w_1)-t_n(w_2)|\le 2^{-h}$.
Conditioning on $\mathcal{F}_h$ then applying \cite[Lemma~10]{Junnila} gives
\begin{equation}\label{eq:DistanceSmallerGood2}
\begin{aligned}
\mathbb{E}\Big[\widetilde{\cB}_{\ell,\gamma,\tg}^{h,\alpha}(w_1)\widetilde{\cB}_{\ell,\gamma,\tg}^{h,\alpha}(w_2)\Big]
&\le \mathbb{E}\Big[ \widetilde{E}_{h,\gamma,\tg}(w_1)\widetilde{E}_{h,\gamma,\tg}(w_2)\Big]2^{(\gamma^2-\frac{(2\gamma-\alpha)^2}{2})(\ell-h)} \,.
\end{aligned}
\end{equation}
\cref{lem:junnila-lemma-7} bounds the  right-hand side of the above display by 
\begin{equation}\label{eq:DistanceSmallerGood3}
\begin{aligned}
    C2^{\gamma^2h}2^{(\gamma^2-\frac{(2\gamma-\alpha)^2}{2})(\ell-h)}
  \le C2^{\gamma^2\log(n)^{1.5}+(\gamma^2-\frac{(2\gamma-\alpha)^2}{2})\ell} \,.
\end{aligned}
\end{equation}

Next, we consider summands in \cref{eq:ScaleDecomp} with $\ell\le h-1$. 
We proceed as in \cref{eq:DistanceSmallerGood1,eq:DistanceSmallerGood2,eq:DistanceSmallerGood3}, this time conditioning on  $\mathcal{F}_{h+1}$ and using $|t_n(w_1)-t_n(w_2)|\ge 2^{-\ell-1} \ge 2^{-h-2}$ in \cite[Lemma~9]{Junnila}: 
\begin{equation}\label{eq:DistanceBiggerGood}
\begin{aligned}
    \mathbb{E}\Big[\widetilde{\cB}_{n,\gamma,\tg}^{h,\alpha}(w_1)\widetilde{\cB}_{n,\gamma,\tg}^{h,\alpha}(w_2)\Big]\le \mathbb{E}\Big[\widetilde{E}_{h+1,\gamma,\tg}(w_1)\widetilde{E}_{h+1,\gamma,\tg}(w_2)\Big] \le C2^{\gamma^2 (h+1)}\le C2^{\gamma^2\log(n)^{1.5}}.
\end{aligned}
\end{equation}

Combining \cref{eq:DistanceSmallerGood1}--\cref{eq:DistanceBiggerGood} yields that, for all $h\le \log(n)^{1.5}$, $j\le Cf(n)^2$, $(w_1,w_2)\in\mathcal{D}_{\ell,j}$, and $\ell \in [3,\ell_0]$,
\begin{equation}\label{eq:AllDistancesGood}
    \mathbb{E}\Big[\widetilde{\cB}_{n,\gamma,\tg}^{h,\alpha}(w_1)
    \widetilde{\cB}_{n,\gamma,\tg}^{h,\alpha}(w_2) \Big]
    \le C2^{\gamma^2\log(n)^{1.5}+(\gamma^2-\frac{(2\gamma-\alpha)^2}{2})\ell}\,.
\end{equation}
Substituting \cref{eq:AllDistancesGood} and \cref{eqn:size-Tn-Inj} into \cref{eq:ScaleDecomp} yields, for $j\leq Cf(n)^2, h\leq \log(n)^{1.5}$ arbitrary 
\begin{multline*}
    \mathsf{Cross}(h,j) \\
    \le C\sum_{\ell=3}^{\ell_0} \sum_{(w_1,w_2)\in\mathcal{D}_{\ell,j}} 2^{-2n}f(n)^{-2} \mathbb{E}\Big[\bar{M}_{n,\ta}^{(1)}(w_1)\bar{M}_{n,\ta}^{(1)}(w_2)\Big] 2^{\gamma^2\log(n)^{1.5}}2^{(\gamma^2-\frac{(2\gamma-\alpha)^2}{2})\ell}\,.
\end{multline*}
Observe for $\ell\le \log(n)^2$ and $(w_1,w_2)\in\mathcal{D}_{\ell,j}$ we have 
\[
|t(w_1)-t(w_2)|\ge 2^{-\log(n)^2-1} \ge 2^{-n}2^{\log(n)^2}
\] for $n$ sufficiently big, so that \cref{Lem:CovarTerm} yields
\begin{align*}
    \mathsf{Cross}(h,j)
    &\leq C\sum_{\ell=3}^{\log(n)^2} |\mathcal{D}_{\ell,j}| \, 2^{-2n}f(n)^{-2} 2^{-\frac{(\alpha-\gamma^2)}{2}\log(n)^2+\gamma^2\log(n)^{1.5}+(\gamma^2-\frac{(2\gamma-\alpha)^2}{2})\ell}\\
    &\qquad+C\sum_{\ell=\log(n)^2+1}^{\ell_0} |\mathcal{D}_{\ell,j}| \, 2^{-2n}f(n)^{-2} n^{-2}2^{\gamma^2\log(n)^{1.5}+(\gamma^2-\frac{(2\gamma-\alpha)^2}{2})\ell} \,.
\end{align*}
Furthermore, by Lemma~\ref{Lem:SubTreeSize}, we have $|\mathcal{D}_{\ell,j}| \le T_n^2 2^{-\ell}= 2^{2n}f(n)^22^{-\ell}$, which yields
\begin{equation}\label{eq:CrossTermControl}
\begin{aligned}
\mathsf{Cross}(h,j) \leq C 2^{-\eta \log(n)^2}\,,
\end{aligned}
\end{equation}
for some $\eta>0$ obtained by choosing $\alpha$ close enough to $\gamma$ so that $\gamma^2-\frac{(2\gamma-\alpha)^2}{2}-1<0$. 

In light of \cref{eq:DiagTerm}, choose also $\alpha$ close enough to $\gamma$ so that $\frac{\gamma^2}{2}+\varepsilon(\alpha)-1<0$, which is possible since $\gamma<\sqrt{2}$. Then substituting \cref{eq:DiagTerm}, \cref{eq:CrossTermControl}, and Lemma~\ref{Lem:SecondMoment} into \cref{eq:ToControllSecondMoment} yields
\begin{align*}
&\mathbb{E} \, \Bigg|\sum_{j=1}^{Cf(n)^2} \sum_{h=0}^{\log(n)^{1.5}}\sum_{w\in\mathcal{N}_{n,j}(v)}  |I_{n}(w)| \, \bar{M}_{n,\ta}^{(1)}(w)\widetilde{E}_{n,\gamma,\tg}(w)\indset{B_{\log(n)^{1.5},n}^{\alpha}(w)}\Bigg|^2 \leq C2^{-\frac{\eta}{2} \log(n)^2}\,.
\end{align*}
We use the last display to conclude \cref{eq:L2PartConv} by Markov's inequality and Borel-Cantelli.
\end{proof}

\subsubsection{Proofs of \cref{Lem:SecondMoment,Lem:FirstMoment,Lem:CovarTerm}} Towards \cref{eq:IntegralLimit}, it remains to prove \cref{Lem:SecondMoment,Lem:FirstMoment,Lem:CovarTerm}.

\begin{proof}[Proof of Lemma \ref{Lem:SecondMoment}]
It is enough to show the bound for $\mathbb{E}[M_{n,\ta}^{(1)}(w)^2]$, since the first moment is bounded by 1 by comparison to $M_{n,\ta}(w)$. We plan to decompose $M_{n,\ta}^{(1)}(w)$ according to the last generation $h$ (counting forward from $n$) at which the field at $t$ was more than $\alpha$ thick. Roughly speaking, the only contribution comes from the value of the field at generation $h$.
For this purpose, we set $h_n := n+h+1$ and
\begin{equation}\label{eq:GoodAndBad}
\begin{aligned}
  A_{h;n}^\alpha(t) &:= \big\{\tS_{(n,n+h],\ta}(t) > \alpha \log(2)h\big\}\,.
\end{aligned}
\end{equation}
This allows us to write
\begin{align*}
M_{n,\ta}^{(1)}(w) &= \lim\limits_{r\to\infty} |I_n(w)|^{-1}\sum_{h = 0}^{\log(n)^2}\int_{I_n(w)} \indset{A_{h;n}^\alpha(t)\cap B_{h,r;n}^\alpha(t)} \widetilde{E}_{(n,n+r],\gamma,\ta}(v_{n+r}(t))\;\mathrm{d}t\\
&= \lim\limits_{r\to\infty} |I_n(w)|^{-1}\sum_{h = 0}^{\log(n)^2} \sum_{u\in\mathcal{N}_{h_n}(w)} \int_{I_{h_n}(u)} \indset{A_{h;n}^\alpha(t) \cap B_{h,r;n}^\alpha(t)} \widetilde{E}_{(n,n+r],\gamma,\ta}(v_{n+r}(t))\;\mathrm{d}t,
\end{align*}
where in the second step we have decomposed $I_n(w)$ into its subintervals in generation $h_n$. Applying Fatou's lemma and Fubini's theorem now yields 
\begin{multline}
\label{eq:SecMomFubini}
  \mathbb{E}\Big[M_{n,\ta}^{(1)}(w)^2 \Big]\le \liminf\limits_{r\to\infty} |I_n(w)|^{-2} \log(n)^2   \\
  \times \max_{h\le \log(n)^2, u\in\mathcal{N}_{h_n}(w)}|\mathcal{N}_{h_n}(w)|^2 \, \mathbb{E}\bigg|\int_{I_{h_n}(u)} \indset{A_{h;n}^\alpha(t) \cap B_{h,r;n}^\alpha(t)} \widetilde{E}_{(n,n+r],\gamma,\ta}(v_{n+r}(t))
  \d t\bigg|^2 \,.
\end{multline}
Let $\mathfrak{T}(w)$ denote the subtree of $\mathfrak{T}$ below $w$. In our embedding of $\mathfrak{T}$ into $[0,1]$, the subtree $\mathfrak{T}(w)$ embeds into the interval $I_n(w)$ of length (roughly \cref{eqn:size-Tn-Inj}) $2^{-n}f(n)^{-1}$.  We now embed $\mathfrak{T}(w)$ into $[0,1]$ by scaling space; this way $(\tS_{(n,n+r]}(t))_{r\in\N, t\in[0,1]}$ is $\log$-correlated in the sense of \cref{rk:junnila-estimates}, and in particular the results of \cite{Junnila} apply. Letting $\lambda_w$ denote the rescaled Lebesgue measure, we have $\lambda_w(I_{h_n})\le 2^{-h-1}$. We now apply  \cite[Proposition 11]{Junnila} to the second moment on the right-hand  of \cref{eq:SecMomFubini} to see that, for all $h\le\log(n)^2$and $u\in\mathcal{N}_{h_n}(w)$,
\begin{multline}\label{eq:DyadLengthIntervalSecMom}
  \mathbb{E}\bigg|\int_{I_{h_n}(u)} \indset{A_{h;n}^\alpha(t)\cap B_{h,r;n}^\alpha(t)} e^{\gamma \tS_{(n,n+r],\ta}(t)}\tZ_{(n,n+r],\gamma,\ta}(t)^{-1}\;\mathrm{d}t\bigg|^2  \\
  \le 2^{-2h}|I_n(w)|^2 \mathbb{E}\bigg[\sup_{t\in I_{h_n}(u)} e^{2\gamma \tS_{(n,n+h],\ta}(t)}\tZ_{(n,n+h],\gamma,\ta}(t)^{-2}\bigg]\,.
\end{multline}
By definition of the model, $\tS_{(n,n+h],\ta}(t) = \tS_{(n,n+h],\ta}(u)$ is constant for all $t\in I_{h_n}(u)$, which together with \cref{lem:junnila-lemma-7} yields
\begin{equation}\label{eq:DyadLengthIntervalSecMomSub}
2^{-2h}|I_n(w)|^2 \mathbb{E}\bigg[\sup_{t\in I_{n+h+1}(u)} 
e^{2\gamma \tS_{(n,n+h],\ta}(t)}\tZ_{(n,n+h],\gamma,\ta}(t)^{-2}
\bigg] \le C2^{-2h+\gamma^2h} |I_n(w)|^2\,.
\end{equation}
 Plugging \cref{eq:DyadLengthIntervalSecMom} and \cref{eq:DyadLengthIntervalSecMomSub}  into \cref{eq:SecMomFubini}  yields
 \begin{align*}
\mathbb{E}\Big[M_{n,\ta}^{(1)}(w)^2 \Big] &\le C\limsup\limits_{r\to\infty}  \log(n)^2 \max_{h\le \log(n)^2} 2^{-2h+\gamma^2h} |\cN_{h_n}(w)|\,.
\end{align*}
The lemma then follows from the bound on $|\cN_{h_n}(w)|$ provided by \cref{eq:FirstStateSubTreeSize}.
\end{proof}

\begin{proof}[Proof of Lemma \ref{Lem:FirstMoment}]
Recall the notation introduced in \cref{eq:GoodAndBad}. A union bound over the set of generations at which the field first becomes more than $\alpha$ thick simply yields
\begin{align*}
M_{n,\ta}^{(2)}(w) \leq  \lim\limits_{r\to\infty} \sum_{\ell=\log(n)^2}^{r} |I_n(w)|^{-1}\int_{I_n(w)} \indset{A_{\ell;n}^{\alpha}(t)}e^{\gamma \tS_{(n,n+r],\ta}(t)}\tZ_{(n,n+r],\gamma,\ta}(t)^{-1}\;\mathrm{d}t\,.
\end{align*}
By generational independence and~\cref{eqn:junnila-lemma-7}, we get 
\[\tZ_{(n,n+r],\hskip-0.252pt\gamma,\ta}(t) \le Ce^{\log(2)\frac{\gamma^2}{2}l}\tZ_{(n+\ell,n+r],\gamma,\ta}(t)\,.\] Moreover, by definition, we have
\[
  \indset{A_{\ell;n}^{\alpha}(t)} \le e^{(\alpha-\gamma) \tS_{\ell,\ta}(t)-(\alpha-\gamma)\alpha\log(2)\ell}\,.
\]
Substituting into the second-to-last display then applying Fatou's lemma, Fubini's theorem, and \cref{lem:junnila-lemma-7} yield
\[
\mathbb{E}\left[M_{n,\ta}^{(2)}(w)\right] \le C\sum_{\ell=\log(n)^2}^\infty 2^{\frac{\alpha^2}{2}\ell-(\alpha-\gamma)\alpha \ell-\frac{\gamma^2}{2}\ell} 
\le C2^{-\frac{(\alpha-\gamma)^2}{2}\log(n)^2}\,.\qedhere
\]
\end{proof}

\begin{proof}[Proof of Lemma \ref{Lem:CovarTerm}]
First, we control the $\mathbb{E}[M_{n,\ta}^{(1)}(w_1)M_{n,\ta}^{(1)}(w_2)]$ term. For this we set $D:=I_{n}(w_1)\times I_{n}(w_2)$. Using Fatou's lemma and Fubini's theorem, and bounding indicators by $1$, we have
\begin{align}\label{eq:CovCalcMainTerm}
&\mathbb{E}[M_{n,\ta}^{(1)}(w_1)M_{n,\ta}^{(1)}(w_2)]  \\
&\le \liminf\limits_{r\to\infty} |D|^{-1} \int_{D} \mathbb{E}\left[e^{\gamma \tS_{(n,n+r],\ta}(t)+\gamma \tS_{(n,n+r],\ta}(s)}\tZ_{(n,n+r],\gamma,\ta}(t)^{-1}\tZ_{(n,n+r],\gamma,\ta}(s)^{-1} \right] \d t \d s \,.\nonumber
\end{align}
We now claim that Lemma~\ref{lem:junnila-lemma-7} implies
\begin{multline}\label{eq:CovCalcCrossTermWrittenOut}
\mathbb{E}\Big[e^{\gamma \tS_{(n,n+r],\ta}(t)+\gamma \tS_{(n,n+r],\ta}(s)}\tZ_{(n,n+r],\gamma,\ta}(t)^{-1}\tZ_{(n,n+r],\gamma,\ta}(s)^{-1} \Big] \\
  =e^{\gamma^2 \sum_{h=1}^{r} \Cov[\ta_{n+h}(t), \ta_{n+h}(s)]+O_n(2^{-n/2})}\,.
\end{multline}
To see how Lemma 2.2 applies here,
recall that $\widetilde{S}_{[n,n+r], \widetilde{a}}(t)= \widetilde{S}_{n+r, \widetilde{a}}(t) - \widetilde{S}_{n, \widetilde{a}}(t)$. Re-defining $\varphi(\gamma_1,\gamma_2)$ in the fourth display of the proof of Lemma 2.2 so that the lowest index of $k$ is $2^n$ instead of $1$ and the highest index of $k$ is $2^{n+r}-1$ instead of $2^n-1$, we have $\varphi(\gamma, \gamma)$ equals the left-hand side of \eqref{eq:CovCalcCrossTermWrittenOut}. Following the remainder of the proof reveals that the resulting $\widetilde{\zeta}(\gamma,\gamma,t,s)$ term is bounded by $O_n(2^{-n/2})$ uniformly in all parameters.

Following the calculation in \cite[Section~3.1, Proof of Theorem 1]{Junnila} until the bound $|R_n(t)| \leq \frac{1}{(n+1)\sin(2\pi t)}$ obtained in the last line of that proof, we similarly have (where their $n$, $t$, and $X_k(x)$ correspond to $2^{n+h}-1$, $t-s$, and $X_k(t)$ respectively for us) for $(t,s)\in D$ (and recalling that $v_n(t) = w_1$, $v_n(s) = w_2$),
\begin{align}\label{eq:CovarBoundGenn+h}
  \Cov[\ta_{n+h}(t), \ta_{n+h}(s)] \le 2^{-n-h} \sin\left(2\pi |t-s| \right)^{-1} 
  \le C2^{-n-h} |t_n(w_1)-t_n(w_2)|^{-1} \,,
\end{align}
where we used the assumption that $|t_n(w_1)-t_n(w_2)|\ge 2^{-n}$, which is much larger than $\max(|t-t_n(w_1)|, |s-t_n(w_2)|)$ in light of \cref{eqn:size-Tn-Inj}.
Plugging \cref{eq:CovCalcCrossTermWrittenOut} and \cref{eq:CovarBoundGenn+h} into \cref{eq:CovCalcMainTerm} yields
\begin{align*}
\mathbb{E}\Big[M_{n,\ta}^{(1)}(w_1)M_{n,\ta}^{(1)}(w_2)\Big] &\le |D|^{-1}\int_{D} e^{C\gamma^2 2^{-n} |t_n(w_1)-t_n(w_2)|^{-1}+O_n(2^{-n/2})} \d t \d s\\
&\le e^{C\gamma^2 2^{-n} |t_n(w_1)-t_n(w_2)|^{-1}+O_n(2^{-n/2})}.
\end{align*}
Thus, we get
\begin{align*}
\mathbb{E}\Big[ \bar{M}_{n,\ta}^{(1)}(w_1)\bar{M}_{n,\ta}^{(1)}(w_2)\Big] \le e^{C\gamma^2 2^{-n}|t_n(w_1)-t_n(w_2)|^{-1}+O_n(2^{-n/2})}-\mathbb{E}[M_{n,\ta}^{(1)}(w_1)]\mathbb{E}[M_{n,\ta}^{(1)}(w_2)].
\end{align*}
From Lemma \ref{Lem:FirstMoment} and $\E[M_{n,\ta}^{(1)}(w)]+\E[M_{n,\ta}^{(2)}(w)] =1$, we have
\[
\mathbb{E}\Big[M_{n,\ta}^{(1)}(w_1)\Big]\mathbb{E}\Big[M_{n,\ta}^{(1)}(w_2)\Big] \ge \Big(1-C2^{-\frac{(\alpha-\gamma)^2}{2}\log(n)^2}\Big)^2 \ge 1-2C\cdot 2^{-\frac{(\alpha-\gamma)^2}{2}\log(n)^2}\,.
\]
Thus, absorbing the factor $2$ into the generic constant $C$,
\begin{align}\label{eq:CovFirstResult}
\mathbb{E}\Big[ \bar{M}_{n,\ta}^{(1)}(w_1)\bar{M}_{n,\ta}^{(1)}(w_1)\Big] &\le e^{\gamma^2 2^{-n}|t_n(w_1)-t_n(w_2)|^{-1}+O_n(2^{-n/2})}-1+C2^{-\frac{(\alpha-\gamma)^2}{2}\log(n)^2} \nonumber\\
&\le C\Big(2^{-n}|t_n(w_1)-t_n(w_2)|^{-1}+2^{-\frac{n}2}+2^{-\frac{(\alpha-\gamma)^2}{2}\log(n)^2}\Big) \,.
\end{align}
Since $f(n)\ge n^2$, we immediately conclude both statements of Lemma~\ref{Lem:CovarTerm}.
\end{proof}

\subsection{Proof of \eqref{it:RNUnifInt}: uniform integrability}
Similarly to the last section we only show that $(\widetilde{R}_{n,\gamma}(t))_{n\in\N}$ is almost surely uniformly integrable with respect to $\widetilde{\mu}_{\gamma,\ta}$, since the other statement can be shown analogously. By the coupling construction, we can control $\widetilde{R}_{n,\gamma}(t)$ for all $t$ which are ``consistently thick'', with an error depending on how late they start being consistently thick. Thus, conceptually we aim to show that points which only start being consistently $(\gamma-\delta)\log(2)$-thick at a late level do not contribute to $\widetilde{\mu}_{\gamma,\ta}$. We split the calculation into the $t$ which are at least $(\gamma-\delta)\log(2)$-thick at level $n$ and those who are not. Recall $\mathcal{T}_{n,\gamma}$ from \eqref{def:time-n-thick}. Say that $t\in \mathcal{T}_{n,\gamma}$ if and only if $v_n(t)\in\mathcal{T}_{n,\gamma}$, and set
\begin{align*}
\widetilde{R}_{n,\gamma,1}(t) := \ind{t\in \mathcal{T}_{n,\gamma}}\widetilde{R}_{n,\gamma}(t)\,,\\
\widetilde{R}_{n,\gamma,2}(t) := \ind{t\in \mathcal{T}_{n,\gamma}^c}\widetilde{R}_{n,\gamma}(t)\,.
\end{align*}
Since $\widetilde{R}_{n,\gamma}(t)= \widetilde{R}_{n,\gamma,1}(t)+ \widetilde{R}_{n,\gamma,2}(t)$, we are done upon showing  that $(\widetilde{R}_{n,\gamma,1}(t))_{n\in\N_0}$, $(\widetilde{R}_{n,\gamma,2}(t))_{n\in\N}$ are both uniformly integrable with respect to $\widetilde{\mu}_{\gamma,\ta}$. 

First, we show uniform integrability of $(\widetilde{R}_{n,\gamma,1}(t))_{n\in\N}$. For this, we fix  $\eta >0$, define
\[
J_{\eta,n} := \int_0^1 \widetilde{R}_{n,\gamma,1}^{1+\eta}\mathrm{d}\widetilde{\mu}_{\gamma,\ta}(t)\,,
\] 
and show that for sufficiently small $\eta$ we have the following $\mathbb{P}$-almost surely:
\begin{equation} \label{eq:GoalUnifInt}
\sup_{n\in\N_0} J_{\eta,n} <\infty\,.
\end{equation}

Since $\widetilde{R}_{n,\gamma,1}(t)$ is constant on $I_n(v)$ for each $v\in\mathcal{N}_n$, we have
\begin{align*}
J_{\eta,n} &= \sum_{v\in\mathcal{N}_n} \widetilde{\mu}_{\gamma,\ta}[I_n(v)] \widetilde{R}_{n,\gamma,1}(t_n(v))^{1+\eta}\\
&\le C\sum_{v\in\mathcal{N}_n}  \ind{v\in \mathcal{T}_{n,\gamma}}\widetilde{\mu}_{\gamma,\ta}[I_n(v)] e^{\gamma(1+\eta)(\tS_{n,\tg}(v)-\tS_{n,\ta}(v))},
\end{align*}
where in the second step, we used \cref{eqn:junnila-lemma-7} to  bound both $\tZ_{n,\gamma,\ta}(t)$ and $\tZ_{n,\gamma,\tg}(t)$ by $Ce^{\log(2)\frac{\gamma^2}{2}n}$.   
 
Now we decompose the sum according to the level $L\in \{1,\dots, n\}$ at which $v$ starts to be consistently thick, i.e., we set
\begin{multline*}
\mathcal{T}_{L,n,\gamma} :=\bigg\{v\in \mathcal{N}_n : \frac{\max\{\tS_{L-1,\ta}(v),\tS_{L-1,\tg}(v)\}}{\log(2)(L-1)} \le \gamma-\delta ,\\
 \forall_{k\in \{L,\dots, n\}}\, \frac{\max\{\tS_{L,\ta}(v),\tS_{L,\tg}(v)\}}{\log(2) k} > \gamma-\delta \bigg\}\,
\end{multline*}
and note that $\bigcup_{L=1}^n \mathcal{T}_{L,n,\gamma} = \mathcal{T}_{n,\gamma}$.   By Lemma \ref{Lem:CouplThick} and Borel-Cantelli, 
there exist $L^\ast <\infty$ $\P$-almost surely such that $\tg_r(v) - \ta_r(v) \leq r^{-2}$ for all $v\in\mathcal{T}_{L,n,\gamma}$ and $r\geq \max\{L^*,L+1\}$. 
Towards \cref{eq:GoalUnifInt}, we from now on fix $\omega\in \Omega^\ast$. We then have
\begin{align*}
J_{\eta,n} &\le C\sum_{L=0}^{L^\ast} \sum_{v\in \mathcal{T}_{L,n,\gamma}} \widetilde{\mu}_{\gamma,\ta}[I_n(v)] e^{\gamma(1+\eta)(\tS_{L^\ast,\tg}(v)-\tS_{L^\ast,\ta}(v))+\gamma(1+\eta)\sum_{r=L^\ast+1}^{n} r^{-2}}\\
&\qquad +C\sum_{L = L^\ast+1}^n \sum_{j\in \mathcal{T}_{L,n,\gamma}} \widetilde{\mu}_{\gamma,\ta}[I_n(v)] e^{\gamma(1+\eta)(\tS_{L,\tg}(v)-\tS_{L,\ta}(v))+\gamma(1+\eta)\sum_{r=L+1}^{n} r^{-2}}\,.
\end{align*}
We can upper bound the first sum in the previous display by
\[
\sup_{v\in\mathcal{N}_{L^\ast}} e^{\gamma(1+\eta)(\tS_{L^\ast,\tg}(v)-\tS_{L^\ast,\ta}(v))}\widetilde{\mu}_{\gamma,\ta}[ [0,1]]\,,
\]
which is $\P$-almost surely finite since $L^\ast$ and $\mathcal{N}_{L^\ast}$ are. Thus, it suffices to prove
\begin{align}\label{eqn:unif-int-goal1}
  \widetilde{J}_{\eta,n} :=\sum_{L = L^\ast+1}^n \sum_{v\in \mathcal{T}_{L,n,\gamma}} \widetilde{\mu}_{\gamma,\ta}[I_n(v)] e^{\gamma(1+\eta)(\tS_{L,\tg}(v)-\tS_{L,\ta}(v))}<\infty \quad \text{$\P$-almost surely.}
\end{align}
For $v\in\mathcal{T}_{L,n,\gamma}$, we have
$\tS_{L-1,\tg}(v) \le (\gamma-\delta)\log(2)L$,
so that
\[
  \widetilde{J}_{\eta,n} \le \sum_{L=L^\ast+1}^n\sum_{v\in\mathcal{T}_{L,n,\gamma}} \widetilde{\mu}_{\gamma,\ta}[I_n(v)] e^{-\gamma(1+\eta)\tS_{L,\ta}(v)+\gamma(1+\eta)\tg_{L}(v)}2^{\gamma(1+\eta)(\gamma-\delta)L}\,.
\]
Next, we note that for $w\in \mathcal{N}_L$ fixed and $v\in\mathcal{N}_n$ with $w \prec v$ we have $\tS_{L,\ta}(v) = \tS_{L,\ta}(w)$, $\tg_{L}(v) = \tg_{L}(w)$ and $\ind{v\in\mathcal{T}_{L,n,\gamma}}\le\ind{w\in \mathcal{T}_{L,\gamma}}$.
Thus, we can bound
\[
\widetilde{J}_{\eta,n}\le \sum_{L=L^\ast+1}^n \sum_{w\in\mathcal{N}_L} \ind{w\in\mathcal{T}_{L,\gamma}} \widetilde{\mu}_{\gamma,\ta}[I_L(w)]e^{-\gamma(1+\eta)\tS_{L,\ta}+\gamma(1+\eta)\tg_{L}(w)}2^{\gamma(1+\eta)(\gamma-\delta)L}.
\]
Now, for $w\in\mathcal{N}_L$, we define
\[
 M_{L,\ta}(w) := |I_{L}(w)|^{-1}\widetilde{\mu}_{\gamma,\ta}[I_{L}(w)]e^{-\gamma \tS_{L,\ta}(w)}\tZ_{L,\gamma,\ta}(w)
\]
and use that $\tZ_{L,\gamma,\ta}(w) \leq C2^{\frac{\gamma^2}{2}\log(L)}$  by \cref{eqn:junnila-lemma-7} to obtain
\begin{multline}
  \widetilde{J}_{\eta,n}  \le \\
C\sum_{L=L^\ast+1}^n \sum_{w\in\mathcal{N}_L}\ind{w\in\mathcal{T}_{L,\gamma}}|I_{L}(w)|\, M_{L,\ta}(w) e^{-\gamma\eta \tS_{L,\ta}(w)+\gamma(1+\eta)\tg_{L}(w)}2^{\gamma(1+\eta)(\gamma-\delta)L-\frac{\gamma^2}{2}L}.\label{eq:UnifIntZwischen}
\end{multline}
Using a union bound and Borel-Cantelli, there exist $c'>0$  and a random variable $L^{\ast\ast}\in (0,\infty)$ $\P$-almost surely
such that for all $L\ge L^{\ast\ast}$ and $w\in\mathcal{N}_{L}$, we have $\tS_{L,\ta}(w)\geq -c'L$. In particular, there exists $\mathcal{C}(\eta)$ taking values in $(0,\infty)$ $\P$-almost surely such that $\gamma \eta(\tS_{L,\ta}(w) +c'L) \geq \mathcal{C}(\eta)$ for all $L\in\N$. Then for any $\varepsilon>0$, we can choose $\eta$ small such that for all $n\in\N$, $L\le n$, $w\in\mathcal{T}_{L,\gamma}$, 
\begin{equation}\label{eq:UnifIntEtaSmall}
\gamma\eta\tS_{L,\ta}(\omega) 
\geq -\varepsilon \log(2) L -\mathcal{C}(\eta) \,.
\end{equation}
Another union bound yields that
\begin{equation}\label{eq:SingleGuySublin}
\gamma(1+\eta)
\max_{L\in\N} L^{-2/3}\log(2)^{-1}\max_{w\in\mathcal{N}_L} |\tg_{L}(w)| =: \mathcal{C}_1<\infty
\end{equation} is $\mathbb{P}$-almost surely finite. Substituting \cref{eq:UnifIntEtaSmall}, \cref{eq:SingleGuySublin} into \cref{eq:UnifIntZwischen} yields that there is a $\mathbb{P}$-a.s.\@ random variable $\mathcal{C}$ such that
\begin{align*}
\widetilde{J}_{\eta,n} &\le \mathcal{C} \sum_{L=L^\ast+1}^n 2^{\gamma(1+\eta)(\gamma-\delta)L+\varepsilon L+\mathcal{C}_1 L^{2/3}-\frac{\gamma^2}{2}L} 2^{-L}f(L)^{-1}\sum_{w\in\mathcal{N}_L} \ind{w\in\mathcal{T}_{L,\gamma}} M_{L,\ta}(w)\\
&=: \mathcal{C} \sum_{L=L^\ast+1}^n \sigma_L,
\end{align*}
where we have also used that $|I_{L}(w)|\le 2^{-L}f(L)^{-1}$. We can upper bound the first moment of the summand by 
\[
\mathbb{E}[\sigma_L]\le  2^{\gamma(1+\eta)(\gamma-\delta)L+\varepsilon L+\mathcal{C}_1 L^{2/3}-\frac{\gamma^2}{2}L} 2^{-\frac{(\gamma-\delta)^2}{2} L} O_L(1) \,.
\]
For $\eta = \varepsilon = 0$ this is bounded by $O_L(L^{-2}2^{-\delta^2 L}2^{\mathcal{C}_1 L^{2/3}}) \le O_L(2^{-\varepsilon' L})$ for some $\varepsilon'>0$. Since the expression is continuous in both $\eta$ and $\varepsilon$, we can choose them so small that $\mathbb{E}[\sigma_L] \le O_L(2^{-\varepsilon'/2 L})$. By Markov's inequality and Borel-Cantelli, this implies that for all but finitely many $L$ we have $\sigma_L\le 2^{-\varepsilon'/4 L }$. This immediately implies \cref{eqn:unif-int-goal1},
finishing the proof of absolute continuity of $\widetilde{R}_{n,\gamma,1}(t)$ with respect to $\widetilde{\mu}_{\gamma,\ta}$. 

We next handle the uniform integrability of $(\widetilde{R}_{n,\gamma,2}(t))_{n\in\N}$ in a very similar manner.
As in the first step above, it suffices to show the $\P$-almost-sure finiteness of
\begin{align}\label{eqn:unif-int-goal2}
  J_{\eta,n}' := \sum_{v\in\mathcal{N}_n}  \ind{t\in \mathcal{T}_{n,\gamma}^c}\widetilde{\mu}_{\gamma,\ta}[I_n(v)] e^{\gamma(1+\eta)(\tS_{n,\tg}(v) - \tS_{n,\ta}(v))} \quad\text{$\P$-almost surely.}
\end{align}
Similar to \cref{eq:UnifIntZwischen}, we have
\[
J_{\eta,n}'\le C\sum_{v\in\mathcal{N}_n} \ind{v\in\mathcal{T}_{n,\gamma}^c} |I_n(v)| M_{n,\ta}(v) e^{-\gamma \eta \tS_{n,\ta}(v)}e^{\gamma(1+\eta) \tS_{n,\tg}(v)}2^{-\frac{\gamma^2}{2}n} \,.
\]
Using \cref{eq:UnifIntEtaSmall} again, taking $L:=n$ there, we find some $\mathcal{C}'(\eta)\in (0,\infty)$ $\P$-almost surely such that 
\[
J'_{\eta,n} \le \mathcal{C}' \sum_{v\in\mathcal{N}_n} \ind{v\in\mathcal{T}_{n,\gamma}^c} |I_n(v)| M_{n,\ta}(v) e^{\varepsilon \log(2) n}e^{\gamma(1+\eta)\tS_{n,\tg}(w)}2^{-\frac{\gamma^2}{2}n}\,.
\] 
We note that by definition,
\[
\ind{v\in\mathcal{T}_{n,\gamma}^c} \le e^{-(\gamma\eta+\delta) \tS_{n,\tg}(v)+(\gamma\eta+\delta)(\gamma-\delta)\log(2)n}\,.
\] 
Using generational independence and $\E[M_{n,\ta}(v)]=1$, the above displays give
\begin{align*}
\mathbb{E}[J_{\eta,n}']&\le 2^{\varepsilon n} 2^{\frac{(\gamma(1+\eta)-\gamma\eta -\delta)^2}{2}n+(\gamma\eta+\delta)(\gamma-\delta)n-\frac{\gamma^2}{2}n}= 2^{\varepsilon n-\frac{\delta^2}{2}n+\gamma^2\eta n -\gamma\eta\delta n}\,.
\end{align*}
By taking $\eta,\varepsilon$ close enough to $0$ this is summable, which implies \cref{eqn:unif-int-goal2}.
This finishes the proof of uniform integrability of $\widetilde{R}_{n,\gamma}(t)$ with respect to $\widetilde{\mu}_{\gamma,\ta}$.

\section{Generalizations and Open Questions}
\label{sec:generalizations-future-directions}
Here, we discuss some natural extensions and open questions raised by our work. To explain what our proof strategy can and can not handle, we frequently refer to the proof outline in \cref{subsec:proof-outline}. In particular, recall from there  the ``discretization step'' and ``coupling step.''

\subsection{Critical chaos}
The construction of the critical multiplicative chaos has not yet appeared for the random Fourier series \eqref{def:trig-field}; indeed, it is left as an open question (and for a more general class of non-Gaussian fields) in \cite[Section~3.3]{Junnila}.
Both the construction of the critical multiplicative chaos as well as absolute continuity with respect to a coupled critical GMC will be the subject of future work.

\subsection{Extension to general bounded domains}
\label{subsec:general-bounded-domains}
In this article, we considered the space $D= [0,1]$ for convenience. 
In this subsection, we discuss the extension of our result to general bounded domains in $\R^d$ (provided an a-priori estimate on the regularity of the eigenfunctions of the Dirichlet Laplacian), where the result (\cref{thm:general}) still holds only for $\gamma$ in the complement of the $L^2$-regime.

Consider a bounded domain $D \subset \R^d$ with $d\geq 2$ (in $d=1$, $D$ must be an interval, and so the situation is easily seen to be equivalent to taking $D=[0,1]$). Let $\lambda_k$ be the $k^{\mathrm{th}}$ smallest eigenvalue (with multiplicity) of $-\Delta$ in $D$ with Dirichlet boundary conditions, and let $e_k$ denote a corresponding eigenfunction, normalized so that $\|e_k\|_{L^2(D)} = 1$. Let $g:=(g_k)_{k\in\N}$ denote a sequence of i.i.d.\ standard Gaussian random variables, and
consider 
\[
  \mathcal{S}_{\infty,g}(x) :=  
  \frac{1}{(2\pi)^d}\sum_{k=1}^\infty g_k \lambda_k^{-d/4}  e_k(x) \qquad \text{where } x \in D\,.
\]
This object
is also known as the log-correlated Gaussian field in $D$
\cite[Proposition~1.1(3)]{LGF-survey}\footnote{To match the description in \cite[Proposition~1.1(3)]{LGF-survey}, note that, as explained in \cite[Section~9]{FGF-survey},  the set $\{\phi_k := \lambda_k^{-d/4}e_k\}_{k\in\N}$  forms an orthonormal basis of the Hilbert space that is the closure of Schwartz functions with respect to the inner-product
$\langle f,g\rangle_d:= \langle(-\Delta)^{d/4}f,(-\Delta)^{d/4}g\rangle_{L^2(D)}$.}. 
One can then consider the associated GMC measure $\mu_{\gamma,g}^D$, as well as the multiplicative chaos associated to the non-Gaussian log-correlated random Fourier series/random wave model
\begin{align}\label{def:general-fourier-series}
  \mathcal{S}_{n,a}(x) := \frac{1}{(2\pi)^d} \sum_{\lambda_k \leq 2^{2n}} a_k \lambda_k^{-\frac{d}{4}}  e_k(x) \qquad \text{for } x \in D\,,
\end{align}
where the $a_k$ again are i.i.d.\ satisfying \cref{eq:CondsOnak}.
The normalization constant and the range $\lambda_k \leq 2^{2n}$ are taken so that
\begin{align}
  \Cov\big(\cS_{n,a}(x), \cS_{n,a}(y)\big) = \log \Big(\frac{1}{d(x,y)} \wedge 2^n\Big) + O(1)\,,
\end{align}
where the $O(1)$ holds uniformly over all $x,y \in D$\footnote{While this formula seems to be forklore as we could not find it explicitly in the literature, we note that one can follow the proof of \cite[Theorem~3]{cgff}, which addresses the $d=2$ case and is written in that paper's Section~4, to obtain the result in any $d\geq 2$.}.
In this form, we can recognize that the critical parameter for the associated GMC is $\gamma = \sqrt{2d}$ and the $L^2$-regime is $(0,\sqrt{d})$. 

Let $\mu_{\gamma,a}^D$ denote the multiplicative chaos associated to the field $(\cS_{n,a}(x))_{x\in D, n\in\N}$. Note that the non-triviality of $\mu_{\gamma,a}^D$ for all $\gamma \in(0,\sqrt{2d})$ follows from the main result of \cite{Junnila} only if the conditions of \cite[Definition~4]{Junnila} are satisfied by the field; in particular, \cite[Eq.~(5)]{Junnila} imposes certain regularity conditions on the $e_k$.

Similarly, as long as the $e_k$ satisfy certain regularity conditions, our proof gives mutual absolute continuity between $\mu_{\gamma,a}^D$ and a coupled GMC measure $\mu_{\gamma,g}^D$ outside of the $L^2$-regime.

\begin{theorem}\label{thm:general}
Define
    \[
      \tau_n := \max_{\lambda_k \in (2^{2(n-1)}, 2^{2n}]}\|\nabla e_k\|_{\infty}^d\,,
    \]
    and suppose the following is satisfied:
    \begin{align}\label{condition-general-theorem}
     \sum_{n\in\N} n^{4d} 2^{-\frac{(\gamma^2-d)}{2}n + \log_2 (\tau_n)} < \infty\,.
    \end{align}
    Then, for $\gamma \in (\sqrt{d}, \sqrt{2d})$, there exists a sequence of i.i.d.\ standard Gaussians $g:= (g_k)_{k\in\N}$ coupled to $a := (a_k)_{k\in\N}$ such that $\mu_{\gamma,g}^D$ and $\mu_{\gamma,a}^D$  are almost surely mutually absolutely continuous.
    \end{theorem}

It would be interesting future work to understand the multiplicative chaos in a domain where the eigenfunctions have poor regularity.

    \begin{remark}\label{rk:full-subcritical-higher-dim}
    Similarly to the outline in Section~\ref{subsec:proof-outline}, we demonstrate below thats the only step which needs $\gamma>\sqrt{d}$ in the proof of Theorem~\ref{thm:general} is the coupling step using the Yurinskii coupling (\cref{Lem:YurCoupl}). In \cite{ChowGang} they explain how the high dimensional coupling they prove suffices to show that for all $d\ge 2$  one can couple $\mu_{\gamma,g}^D$ and $\mu_{\gamma,a}^D$ to be almost surely mutually absolutely continuous when the domain is taken to be $D = (0,1]^d$. The restriction on the domain comes from the fact that their CLT is not covariance agnostic, but requires control on the spectral norm of the covariance matrix of $(\widetilde{a}_n(v))_{v\in \mathcal{K}}$ from \cref{def:tree-field-increments}, where $\mathcal{K}\subseteq \{1,\dots, T_n\}$ is arbitrary with $|\mathcal{K}|\le 2^n$. Their estimate on the spectral norm crucially uses that the eigenfunctions of the Dirichlet Laplacian on unit cubes are given by products of signs and cosines. Thus, extending Theorem~\ref{thm:general} to all subcritical $\gamma$ in the full generality of domains considered remains open.
    \end{remark}

We now briefly describe how the proof of \cref{thm:general} differs from the proof of \cref{Theo:Main} with the aim of explaining the regularity condition \eqref{condition-general-theorem}.

The proof of \cref{Theo:Main} was outlined in \cref{subsec:proof-outline}, and we use the terminology of this outline here. Condition \eqref{condition-general-theorem} is needed for exactly one step: the discretization step (we are able to move to a discrete hierarchical field by using the regularity of the eigenfunctions). We seek to define a nested sequence of partitions of $D$ as well as an approximation of $\cS_{n,a}$ that is piecewise constant on each piece of the $n^{\mathrm{th}}$ partition, so that the approximation can be considered as a  tree-indexed discrete field.
Define the increments
\[
  \mathcal{X}_n(x) := \cS_{n,a}(x) - \cS_{n-1,a}(x) = \sum_{\lambda_k \in (2^{2(n-1)}, 2^{2n}]} a_k \lambda_k^{-\frac{d}4}e_k(x)\,.
\] 
For two points $x, y\in D$, we have
\[
  \mathcal{X}_n(x) -\mathcal{X}_n(y) \asymp (2^{2n})^{-\frac{d}4} \sum_{\lambda_k \in (2^{2(n-1)}, 2^{2n}]} a_k \lambda_k^{-\frac{d}4} (e_k(x) -e_k(y))\,.
\]
Recall that Weyl's law states $\lambda_k \asymp k^{2/d}$, so the above sum has $\asymp (2^{2n})^{d/2}$ terms. From the variance of the right-hand side above, we see that 
\[
  \mathcal{X}_n(x) - \mathcal{X}_n(y) \lesssim  \max_{\lambda_k \in (2^{2(n-1)}, 2^{2n}]} |e_k(x) -e_k(y)| \leq \max_{\lambda_k \in (2^{2(n-1)}, 2^{2n}]}\|\nabla e_k\|_{\infty} \|x-y\|\,.
\]
Thus, $\mathcal{X}_n(x) - \mathcal{X}_n(y)$ becomes small once $\|x-y\| \ll \tau_n^{-1/d}$.
From this heuristic calculation,  we see that, in order to make the chaining argument underlying \cref{Prop:TrigToHier} work, the $n^{\mathrm{th}}$ partition in our nested sequence of partitions of $D$ should contain (slightly more than) $\tau_n$ pieces of diameter (slightly less than) $\tau_n^{-1/d}$; in particular, in analogy with the bound on $T_n$ from \cref{eqn:size-Tn-Inj}, $n^{4d} \tau_n$ pieces in the $n^{\mathrm{th}}$ partition will suffice. Thus, following the proof of \cref{Prop:TrigToHier} allows us to move to a discrete model on a tree where the set of particles $\cN_n$ in the $n^{\mathrm{th}}$ generation satisfies $|\cN_n| \leq  n^{4d} \tau_n$.

Next is the coupling step. Following the proof of \cref{Lem:YurCoupl} with the above bound on $|\cN_n|$, the Yurinskii coupling (\cref{Lem:YrCoupSource}) yields that the right-hand of \cref{eq:CouplingGoodOnThick} should be replaced by
\[
  (2^{2n})^{-\frac{3d}{4}} \cdot |\text{$\gamma$-thick points}| \cdot (2^{2n})^{\frac{d}2} \leq (2^{2n})^{-\frac{3d}{4}} \cdot |\cN_n| \cdot 2^{-\frac{\gamma^2}{2}n} \cdot (2^{2n})^{\frac{d}2} = n^{4d} 2^{-\frac{(\gamma^2-d)}{2}n + \log_2 \tau_n}\,.
\]
We need the right-hand above to be summable in $n$; this follows from condition~\eqref{condition-general-theorem} when $\gamma > \sqrt{d}$.

After the discretization and the coupling step, the proof proceeds exactly as in the proof of \cref{Theo:Main}, since we have moved completely to the comparison of two chaoses associated to  discrete models on the same tree (i.e., the space $D$ is no longer relevant).

\subsection*{Acknowledgements}
Both authors thank Ofer Zeitouni for suggesting the problem and for many helpful discussions. YHK thanks Klara Courteaut for a helpful discussion on the log-correlated Gaussian field in general domains and for pointing us to \cite{cgff}.
We thank the anonymous referee for a careful reading of the paper and for
 comments and detailed suggestions which helped us to improve
the manuscript.

\appendix
\section{Chaos measures are supported on their thick points} \label{appendix:thick}
In this section we provide a proof for Lemma~\ref{Lem:SupportThick}. This is not novel, see for example \cite[Theorem~4.1]{RV14} for a similar argument in the Gaussian case. However, we did not find the result written down explicitly in the non-Gaussian case.

\begin{proof}[Proof of Lemma~\ref{Lem:SupportThick}]
We first show  $\widetilde{\mu}_{\gamma,\widetilde{a}}[\mathcal{T}^c] = 0$.
For this purpose, we define for $j\in \N$ the sets
\begin{align*}
E_{n,j}^+ &:= \left\{t\in [0,1] : \widetilde{S}_{n,\widetilde{a}}(t) > (\gamma+1/j)\log(2)n\right\}\,,\\
E_{n,j}^- &:= \left\{t\in [0,1]: \widetilde{S}_{n,\widetilde{a}}(t) <(\gamma-1/j)\log(2)n \right\}\,.
\end{align*}
By definition, we have
\[
\mathcal{T}^c \subseteq \bigcup_{j\in\N} \limsup_{n\to\infty} (E_{n,j}^+\cup E_{n,j}^-) \,.
\]
Thus, it is enough to show that $\mathbb{P}$-a.s.\ $\widetilde{\mu}_{\gamma,\widetilde{a}}[\limsup_{n\to\infty} (E_{n,j}^{+}\cup E_{n,j}^{-})] = 0$ for all $j\in\N$. We show this using the Borel-Cantelli lemma for the measure $Q$  on $\Omega \times [0,1]$ defined by $Q[A] := \mathbb{E}[\int \indset{A} \widetilde{\mu}_{\gamma,\widetilde{a}}(dt)]$, so that it suffices to show that for all $j\in\N$,
\begin{equation}\label{eq:NessCondBorCant}
\sum_{n=1}^\infty \mathbb{E}\Big[ \widetilde{\mu}_{\gamma,\widetilde{a}}[E_{n,j}^+]+\widetilde{\mu}_{\gamma,\widetilde{a}}[E_{n,j}^-]\Big] <\infty \,.
\end{equation}
Since $\indset{E_{n,j}^+}(t)$ is constant on $I_n(w)$ for $w\in \cN_n$, we have
\begin{align*}
\mathbb{E}\bigg[ \widetilde{\mu}_{\gamma,\widetilde{a}}[E_{n,j}^+] \bigg] &= \mathbb{E}\bigg[ \sum_{w\in \cN_n} \widetilde{\mu}_{\gamma,\widetilde{a}}[I_n(w)] \ind{\widetilde{S}_{n,\widetilde{a}}(w) > (\gamma+1/j)\log(2)n}\bigg]\\
&= \sum_{w\in \cN_n} \mathbb{E}\bigg[ |I_n(w)|\ind{\widetilde{S}_{n,\widetilde{a}}(w)>(\gamma+1/j)\log(2)n} e^{\gamma \widetilde{S}_{n,\widetilde{a}}(w)-\frac{\gamma^2}{2}\log(2)n}\bigg] \,,
\end{align*}
where in the second step we have used generational independence to get rid of the contribution of all but the first $n$ generations to the above expectation. 
Next, we use the bound 
\begin{equation}\label{eq:TiltedExpCheb1}
\ind{\widetilde{S}_{n,\widetilde{a}}(w) >(\gamma+1/j)\log(2)n} \le e^{ j^{-1}\widetilde{S}_{n,\widetilde{a}}(w)-j^{-1}(\gamma+1/j) \log(2)n}
\end{equation}
together with \cref{lem:junnila-lemma-7} to obtain
\begin{equation} \label{eq:EN+Bound}
\mathbb{E}\Big[ \widetilde{\mu}_{\gamma,\widetilde{a}}[E_{n,j}^+] \Big]\le \sum_{w\in \cN_n} |I_n(w)| 2^{\frac{(\gamma+j^{-1})^2}{2}n-j^{-1}(\gamma+j^{-1})n-\frac{\gamma^2}{2}n} = 2^{-\frac{n}{2j^2}}\sum_{w\in \cN_n} |I_n(w)| = 2^{-\frac{n}{2j^2}} \,.
\end{equation}
An analogous argument, using the bound
\begin{equation*}
\ind{\widetilde{S}_{n,\widetilde{a}}(w)<(\gamma-1/j)\log(2)n} \le e^{j^{-1}(\gamma-1/j)\log(2)n-j^{-1}\widetilde{S}_{n,\widetilde{a}}(w)}
\end{equation*}
instead of \cref{eq:TiltedExpCheb1} yields
\begin{equation}\label{eq:EN-Bound}
\mathbb{E}\Big[ \widetilde{\mu}_{\gamma,\widetilde{a}}[E_{n,j}^-]\Big] \le 2^{-\frac{n}{2j^2}}.
\end{equation}
Combining \cref{eq:EN+Bound}, \cref{eq:EN-Bound} yields \cref{eq:NessCondBorCant}. That  $\widetilde{\mu}_{\gamma,\tg}[\mathcal{T}^c] = 0$ follows by taking $\ta := \tg$  above.
\end{proof} 
\section{Properties of our partitions}
\begin{lemma}\label{Lem:SubTreeSize}
Fix $n,h \in \N$ and $w\in \mathcal{N}_{n}$. Set $\mathcal{N}_{n+h}(w):= \{v\in\mathcal{N}_{n+h} : v_n = w\}$. We have
\begin{equation}\label{eq:FirstStateSubTreeSize}
|\mathcal{N}_{n+h}(w)|\le \frac{f(n+h)}{f(n)}2^{h+1}.
\end{equation}
Furthermore, we set $\mathcal{D}_{I,n}(v) := \left\{ w\in \mathcal{N}_n(v) : |t_n(w)-t_n(v)| \in I\right\}$ for $v\in\mathcal{N}_n$ and $I\subseteq \R$. We have
\begin{equation}\label{eq:SecStateSubTreeSize}
|\mathcal{D}_{I,n}| \le T_n\, |I|\,.
\end{equation}
\end{lemma}
\begin{proof}
We first prove \cref{eq:FirstStateSubTreeSize}. By construction, we have
\begin{align*}
|\mathcal{N}_{n+h}(w)| \le \sum_{k= n+1}^{n+h} |I_n(w)| \, 2^kf(k) \le |I_n(w)| 2^{n+h+1}f(n+h) \,.
\end{align*}
The bound on $|I_n(w)|$ provided by \cref{eqn:size-Tn-Inj} yields \cref{eq:FirstStateSubTreeSize}.
We show \cref{eq:SecStateSubTreeSize} similarly: 
\begin{equation*}
|\mathcal{D}_{I,n}(v)|\le \sum_{k=1}^n |I+t_n(v)|\, 2^kf(k) \le |I|\, 2^{n+1}f(n)  = T_n\, |I|\,.  \qedhere
\end{equation*}
\end{proof}

\printbibliography

\end{document}